\documentclass[a4paper,11pt]{amsart}
\usepackage{amsmath,amsthm,amssymb,enumitem}
\usepackage[T1]{fontenc}
\usepackage{lmodern}
\usepackage{microtype}
\usepackage{hyperref}
\hypersetup{hidelinks,pdftitle={Oka realizations, fundamental groups, and holomorphic cofibrations},pdfauthor={Bin Guo, Yuta Kusakabe, Peng-Chao Wang}}
\newtheorem{theorem}{Theorem}[section]
\newtheorem{proposition}[theorem]{Proposition}
\newtheorem{lemma}[theorem]{Lemma}
\newtheorem{corollary}[theorem]{Corollary}
\theoremstyle{definition}

\newtheorem{problem}[theorem]{Problem}
\newtheorem{remark}[theorem]{Remark}

\newcommand{\C}{\mathbb C}
\newcommand{\R}{\mathbb R}
\newcommand{\Z}{\mathbb Z}
\newcommand{\Q}{\mathbb Q}

\newcommand{\CP}{\mathbb{CP}}
\newcommand{\GL}{\mathrm{GL}}
\newcommand{\SL}{\mathrm{SL}}
\newcommand{\Conf}{\operatorname{Conf}}

\begin{document}
\title[Oka realizations and holomorphic cofibrations]{Oka realizations, fundamental groups,\newline and holomorphic cofibrations}
\author{Bin Guo}
\address[Bin Guo]{Academy of Mathematics and Systems Science, Chinese Academy of Sciences, Beijing 100190, China}
\email{guobin181@mails.ucas.ac.cn}
\author{Yuta Kusakabe}
\address[Yuta Kusakabe]{Faculty of Mathematics, Kyushu University, Fukuoka 819-0395, Japan}
\email{kusakabe@math.kyushu-u.ac.jp}
\author{Peng-Chao Wang}
\address[Peng-Chao Wang]{Academy of Mathematics and Systems Science, Chinese Academy of Sciences, Beijing 100190, China}
\email{wangpengchao26@mails.ucas.ac.cn}
\date{}

\begin{abstract}
The homotopy types of Oka manifolds are exactly those of countable finite-dimensional CW complexes. We construct realizations as open subsets of complex Euclidean spaces with explicit dimension bounds. Other realizations have good open covers by sets biholomorphic to $\C^N$, with prescribed finite nerves, or are unions of Fatou--Bieberbach domains whose nonempty finite intersections are again Fatou--Bieberbach domains. Every nonempty Stein manifold admits a proper holomorphic embedding into an Oka open subset of some $\C^N$, and this embedding is a homotopy equivalence. We obtain holomorphic representatives of ordinary and generalized cohomology classes, and realize every countable group as the fundamental group of an Oka domain in $\C^5$. The fundamental groups of algebraically elliptic manifolds are precisely the finite groups, and every virtually abelian K\"ahler group is the fundamental group of a projective Oka manifold. Finally, we classify holomorphic cofibrations in L\'arusson's intermediate model structures and establish factorizations over aspherical targets into proper holomorphic embeddings and holomorphic fibre bundles with Oka fibre.
\end{abstract}
\maketitle

\section{Introduction}

Gromov asked whether every finite CW complex is homotopy equivalent to an Oka manifold \cite[0.7.B$''$]{Gro89}. This realization problem relates the flexibility of holomorphic maps to the topology of their targets. We solve it for all countable finite-dimensional CW complexes and study realizations with additional geometric structure, their fundamental groups, and the corresponding questions for holomorphic maps.

Throughout the paper, complex manifolds are second countable, have a fixed finite complex dimension, and may be disconnected. For topological spaces $X$ and $T$, let $\mathcal C(X,T)$ be the space of continuous maps with the compact-open topology. For complex manifolds $X$ and $Y$, let $\mathcal O(X,Y)\subset\mathcal C(X,Y)$ be the subspace of holomorphic maps.

An \emph{Oka manifold} is a complex manifold $Y$ satisfying the \emph{convex approximation property}: for every $m\ge1$ and every compact convex set $K\subset\C^m$, each holomorphic map from a neighbourhood of $K$ to $Y$ can be approximated uniformly on $K$ by holomorphic maps $\C^m\to Y$ \cite[Defs.~2.1 and~2.3]{For13}. Equivalently, maps from Stein manifolds to $Y$ satisfy the parametric Oka property with approximation and interpolation \cite[Thm.~5.4.4]{For17}. In particular, for every Stein manifold $X$, the inclusion $\mathcal O(X,Y)\hookrightarrow\mathcal C(X,Y)$ is a weak homotopy equivalence. A local approximation characterization of Oka manifolds is given in \cite{ADGWX26}. An \emph{Oka realization} of a space is an Oka manifold homotopy equivalent to that space.

\subsection*{Realizations and their geometry}

\begin{theorem}[Oka realizations]\label{intro:countable}
A homotopy type has an Oka realization if and only if it is the homotopy type of a countable finite-dimensional CW complex. Every such homotopy type is realized by an open subset of some $\C^N$ that is the union of countably many Fatou--Bieberbach domains whose nonempty finite intersections are again Fatou--Bieberbach domains.
\end{theorem}

A \emph{Fatou--Bieberbach domain} is a proper subdomain of $\C^N$ biholomorphic to $\C^N$. The proof in Section~\ref{sec:finite-nerve} uses products of pairwise disjoint Fatou--Bieberbach domains to encode the simplices of a subdivision. The nerve of the resulting cover is homotopy equivalent to the given CW complex, and the nerve lemma gives the same homotopy type for the union. Each point has an open neighbourhood biholomorphic to $\C^N$, so the manifold is strongly dominable and hence Oka by the theorem of An, Du, Guo, Wang and Xie \cite[Thm.~1.3]{ADGWX26}.

A second construction gives an explicit ambient-dimension bound. If $T$ is a nonempty countable CW complex of dimension $d$, Theorem~\ref{thm:low-dimension} realizes it by an Oka open set in $\C^{N(d)}$, where
\[
N(d)=2\left\lfloor\frac d2\right\rfloor+3
=\begin{cases}d+3,&d\text{ even},\\d+2,&d\text{ odd}.
\end{cases}
\]
This construction starts with a Euclidean open set having the required homotopy type and applies a real-linear change of coordinates depending on one complex parameter. Every point of the resulting open set $D$ has a neighbourhood $V\subset D$ and finitely many holomorphic maps $h_j:V\times\C\to D$ with $h_j(y,0)=y$, such that $\{\partial_t h_j(y,0)\}_j$ spans $T_yD$ for every $y\in V$. The criterion for families of entire curves \cite[Cor.~4.9]{ADGWX26}, recalled in Theorem~\ref{thm:entire-families}, then gives the Oka property.

We also realize each finite simplicial complex as the nerve of a good open cover by sets biholomorphic to $\C^D$. If $L$ has $v$ vertices and $s$ maximal simplices, Theorem~\ref{thm:prescribed-nerve} produces a complex manifold of dimension $5s+2$ with a good open cover by exactly $v$ sets biholomorphic to $\C^{5s+2}$ and with nerve $L$. The construction glues copies of $\C^{5s+2}$ by holomorphic transition maps that translate the fibre coordinates. Properness of the differences of the translating maps ensures that the quotient is Hausdorff. For a general countable $d$-dimensional CW complex, the same method gives a realization of dimension $5d+7$ with a countable good open cover by sets biholomorphic to $\C^{5d+7}$.

The dimension bounds can be compared as follows; $d$ is the CW dimension and $s$ is the number of maximal simplices of a finite complex.
\begin{center}
\begin{tabular}{@{}p{.34\textwidth}p{.17\textwidth}p{.40\textwidth}@{}}
Realization & Dimension & Result \\
\hline
Stein manifold & $d$ & Remark~\ref{rem:stein-types} \\
Oka open subset of $\C^N$ & $N(d)$ & Theorem~\ref{thm:low-dimension} \\
Countable good cover by copies of $\C^N$ & $5d+7$ & Corollary~\ref{nerve:countable} \\
Prescribed finite nerve & $5s+2$ & Theorem~\ref{thm:prescribed-nerve}
\end{tabular}
\end{center}

\subsection*{Homotopy types and holomorphic representatives of cohomology classes}

Theorem~\ref{intro:countable} reduces Oka realization to the existence of a countable finite-dimensional CW model. In particular, every Moore space $M(G,n)$ constructed from a countable free resolution of a countable abelian group $G$, with $n\ge1$, has an Oka realization. An Eilenberg--MacLane space $K(G,n)$ is a connected CW complex with $\pi_n\cong G$ and $\pi_j=0$ for $j\ne n$, $j\ge1$. The integral cohomological dimension $\operatorname{cd}_{\Z}G$ is the projective dimension of the trivial $\Z[G]$-module $\Z$. The following corollary characterizes the Eilenberg--MacLane spaces that have Oka realizations.

\begin{corollary}\label{intro:consequences}
\leavevmode
\begin{enumerate}[label=\textup{(\roman*)}]
\item For a group $G$, $K(G,1)$ has an Oka realization if and only if $G$ is countable and $\operatorname{cd}_{\Z}G<\infty$.
\item For a nonzero abelian group $G$ and $n\ge2$, $K(G,n)$ has an Oka realization if and only if $G$ is a finite-dimensional rational vector space and $n$ is odd.
\end{enumerate}
\end{corollary}

The proofs are given in Corollaries~\ref{cor:aspherical} and~\ref{cor:higher-em}. For every $d\ge1$ and every countable CW complex $T$, realizing the $d$-skeleton gives an Oka manifold $Y$ and a $d$-connected map $Y\to T$ (Corollary~\ref{cor:finite-skeleta-oka}). We use the convention that a $d$-connected map induces a bijection on $\pi_0$, isomorphisms on $\pi_i$ for $0<i<d$, and a surjection on $\pi_d$ at every base point. For Stein manifolds of bounded dimension, these maps give holomorphic representatives of cohomology classes.

\begin{theorem}[holomorphic representatives of cohomology classes]\label{intro:em}
Let $E$ be a spectrum with countable homotopy groups, $k\in\Z$, and $m\ge0$. There are an Oka open set $Y\subset\C^N$ and a class $u\in E^k(Y)$ such that, for every Stein manifold $X$ of complex dimension at most $m$, pullback gives a bijection
\[
\pi_0\mathcal O(X,Y)\longrightarrow E^k(X),\qquad [g]\longmapsto g^*u,
\]
natural with respect to holomorphic maps between Stein sources.
\end{theorem}

Here $\pi_0\mathcal O(X,Y)$ is the set of path components of $\mathcal O(X,Y)$: two maps represent the same class when they are joined by a continuous family of holomorphic maps $X\to Y$. For any countable CW complex $T$ and fixed $m,r\ge0$, Proposition~\ref{prop:hol-rep} gives an Oka open set $Y\subset\C^N$ and a continuous map $f:Y\to T$ such that, for every Stein manifold $X$ of complex dimension at most $m$, the map $\mathcal O(X,Y)\to\mathcal C(X,T)$, $g\mapsto f\circ g$, is bijective on $\pi_0$ and induces isomorphisms on $\pi_i$ for $1\le i\le r$ at every base point. Theorem~\ref{intro:em} follows in Corollary~\ref{cor:theories}; it applies to ordinary cohomology with countable coefficients, stable cohomotopy, complex cobordism, and real and complex $K$-theory. A single pair $(Y,u)$ gives the bijections for every Stein manifold if and only if the representing space $E_k=\Omega^\infty\Sigma^kE$ has an Oka realization. Related constructions of Oka complements enter Benoist's work on the \'etale cohomology of Stein algebras \cite[\S\S2--3]{Ben26}.

\subsection*{Fundamental groups}

Every countable group occurs as the fundamental group of an Oka domain in $\C^5$ (Corollary~\ref{cor:countable-groups}). The domain can moreover contain a properly embedded Stein surface whose inclusion is a homotopy equivalence. Additional geometric conditions lead to more restrictive classifications.

\begin{theorem}[algebraically elliptic manifolds]\label{intro:algebraic}
The fundamental groups of algebraically elliptic manifolds are exactly the finite groups. Every finite group is the fundamental group of a flexible quasi-affine manifold, and every finite cyclic group is the fundamental group of such a surface. For every finite group $G$ and every $d\ge1$, there is a $d$-connected map from an algebraically elliptic quasi-affine manifold to the classifying space $K(G,1)$.
\end{theorem}

An algebraically elliptic manifold admits a dominating algebraic spray; its underlying complex manifold is Oka. The finite-group and cyclic-surface constructions originate in the second author's work \cite{Kus22}. Section~\ref{sec:algebraic-groups} also gives configuration-space realizations and, for every $d\ge1$, $d$-connected maps from algebraically elliptic manifolds to $K(G,1)$ for finite $G$. Algebraic ellipticity is equivalent to algebraic subellipticity by \cite{KZ23}. Theorem~\ref{thm:elliptic-finite-groups}, Proposition~\ref{grp:cyclic-surfaces}, and Corollary~\ref{cor:algebraic-bg} prove Theorem~\ref{intro:algebraic}.

A group is \emph{virtually abelian} if it has an abelian subgroup of finite index. A \emph{K\"ahler group} is the fundamental group of a compact K\"ahler manifold.

\begin{theorem}[compact K\"ahler Oka manifolds]\label{intro:projective}
A group is the fundamental group of a compact K\"ahler Oka manifold if and only if it is a virtually abelian K\"ahler group. Every such group is the fundamental group of a projective Oka manifold.
\end{theorem}

The existence theorem in Section~\ref{sec:kahler} constructs simply connected projective Oka manifolds with free finite-group actions and combines them with the affine actions of Baues--Riesterer \cite{BR11}. The necessity in Theorem~\ref{intro:projective} uses Campana's criterion \cite[preprint, Cor.~8.11]{Cam04} and the abelianity theorem in the recent preprint \cite[Thm.~1.1]{OAI26}. The constructive realization theorem is stated separately as Theorem~\ref{thm:va-full}.

\subsection*{Holomorphic cofibrations and factorizations}

A connected manifold is \emph{aspherical} if its universal cover is contractible.

In L\'arusson's fine intermediate model structure, the presheaves associated with Stein manifolds are cofibrant, and those associated with Oka manifolds are fibrant \cite{Lar04,Lar05}. The question whether Stein inclusions and biholomorphisms exhaust the holomorphic cofibrations appears in \cite[p.~32]{FL11} and in L\'arusson's appendix to \cite[p.~802]{For13}. Theorem~\ref{cof:main} gives the following classification.
\begin{quote}
A holomorphic map $i:A\to X$ is an intermediate cofibration if and only if it is a biholomorphism or a closed holomorphic embedding with Stein target.
\end{quote}
The statement uses L\'arusson's Stein site, whose objects include disconnected Stein manifolds. Its proof uses pushouts of the generating cofibrations and shows that the target of every holomorphic cofibration that is not a biholomorphism is a holomorphic retract of a Stein manifold.

Every nonempty Stein manifold $S$ admits a proper holomorphic embedding $j:S\to D$, where $D\subset\C^N$ is Oka and $j$ is a homotopy equivalence (Theorem~\ref{thm:stein-open-embedding}). This answers the question about embeddings into Oka manifolds recorded in \cite[\S1.2]{DeV20}. Ritter established such embeddings for open Riemann surfaces with elliptic targets and studied the additional requirement that the target be Stein \cite[\S4]{Rit13}.

Theorem~\ref{thm:factorization} gives a relative version: every holomorphic map from a nonempty Stein manifold to an aspherical complex manifold factors as a proper holomorphic embedding that is a homotopy equivalence followed by a holomorphic fibre bundle with Oka fibre. For a nullhomotopic holomorphic map from a nonempty Stein manifold to a connected complex manifold, a factorization into such an embedding followed by a holomorphic Serre fibration exists if and only if the target is aspherical (Corollary~\ref{cof:aspherical-factorization}). Section~\ref{sec:factorizations} relates these geometric factorizations to intermediate cofibrations.

Sections~\ref{sec:finite-nerve} and~\ref{sec:prescribed-nerves} give the three realization constructions. Section~\ref{sec:applications} develops the homotopy and cohomology consequences. Fundamental groups are treated in Section~\ref{sec:groups}, and Section~\ref{sec:factorizations} proves the cofibration classification and the relative factorization theorem.

\section{Oka open subsets of complex Euclidean spaces}\label{sec:finite-nerve}

We first give Oka realizations with covers by Fatou--Bieberbach domains and then construct Oka open subsets of $\C^N$ with an explicit bound on $N$. We use two criteria for the Oka property from \cite{ADGWX26} and a realization theorem for countable CW complexes.

A complex manifold $Y$ of dimension $n$ is \emph{strongly dominable} if, for every $y\in Y$, there is a holomorphic map $f:\C^n\to Y$ with $f(0)=y$ and $df_0$ invertible.

\begin{theorem}[{\cite[Thm.~1.3]{ADGWX26}}]\label{thm:strong-dom}
A complex manifold is Oka if and only if it is strongly dominable.
\end{theorem}

We also use the following criterion.
\begin{theorem}[{\cite[Cor.~4.9]{ADGWX26}}]\label{thm:entire-families}
Let $Y$ be a complex manifold. Suppose that every point of $Y$ has an open neighbourhood $V$ and finitely many holomorphic maps
\[
h_j:V\times\C\longrightarrow Y,\qquad h_j(y,0)=y,
\]
such that the tangent vectors $\partial_t h_j(y,0)$ span $T_yY$ at every $y\in V$. Then $Y$ is Oka.
\end{theorem}

\begin{lemma}[{\cite[\S2, Fact~2 and its proof]{KS10}}]\label{lem:topological-open}
Every nonempty countable CW complex of dimension $d$ is homotopy equivalent to an open subset of $\R^{2d+1}$.
\end{lemma}

Indeed, one can replace the complex by a countable locally finite simplicial complex of the same dimension, properly embed it in $\R^{2d+1}$, and take an open regular neighbourhood. We use the nerve lemma for open covers of paracompact spaces \cite[Cor.~4G.3]{Hat02}. A \emph{good open cover} is an open cover all of whose nonempty finite intersections are contractible. The \emph{nerve} of an indexed open cover $(U_i)_{i\in I}$ is the simplicial complex whose simplices are the nonempty finite sets $J\subset I$ such that $\bigcap_{j\in J}U_j\ne\varnothing$.

\subsection{Fatou--Bieberbach covers}
A simplicial complex is \emph{flag} if every finite set of pairwise adjacent vertices spans a simplex. A \emph{nonedge} is a pair of distinct vertices not joined by an edge. There are infinitely many pairwise disjoint Fatou--Bieberbach domains $D_1,D_2,\ldots\subset\C^2$ \cite[Example~9.5]{RR88}. Put $D_0=\C^2$.

The proof of Theorem~\ref{intro:countable} assigns a finite tuple of labels to each vertex of a barycentric subdivision $L$ and associates a product of the domains $D_j$ to that tuple. Two products associated with distinct vertices intersect exactly when those vertices are adjacent in $L$.\footnote{The corresponding interval-graph representation assigns $\R$ to label $0$ and pairwise disjoint intervals to the positive labels in each coordinate. Conditions (A) and (B) recover the $1$-skeleton of $L$ as the intersection of these interval graphs; compare the definition of boxicity in \cite{Rob69}.}

\begin{proof}[Proof of Theorem~\ref{intro:countable}]
Every complex manifold has the homotopy type of a countable finite-dimensional CW complex \cite[Prop.~3]{KS10}. This proves necessity.

For the realization, let $T$ be a countable CW complex of dimension $n$. It is homotopy equivalent to a countable simplicial complex $P_0$ of dimension $n$ \cite[Thm.~2C.5]{Hat02}. Let $P$ be the barycentric subdivision of $P_0$, and colour each vertex of $P$, that is, each simplex of $P_0$, by its dimension in $\{0,\dots,n\}$. The vertices of a simplex $\sigma$ of $P$ have distinct colours; let $t(\sigma)$ be the set of these colours, and for $C\subseteq t(\sigma)$ let $\sigma|_C$ be the face of $\sigma$ spanned by its vertices with colours in $C$.

Let $L$ be the barycentric subdivision of $P$: its vertices are the nonempty simplices of $P$, and its simplices are the chains under inclusion. Then $L$ is flag, $|L|\cong|P|\simeq T$, and a nonedge of $L$ is a pair of simplices of $P$ neither of which contains the other. Fix an injection $\iota$ of the set of simplices of $P$ into $\{1,2,\dots\}$, and use the following labellings of the vertices of $L$ by nonnegative integers:
\begin{enumerate}[label=\textup{(\roman*)}]
\item for each nonempty $C\subseteq\{0,\dots,n\}$, let $\varphi_C(\sigma)=\iota(\sigma|_C)$ if $C\subseteq t(\sigma)$, and $\varphi_C(\sigma)=0$ otherwise;
\item for each pair $\{S,S'\}$ of subsets of $\{0,\dots,n\}$ neither of which contains the other, let $\psi_{S,S'}(\sigma)=\iota(\sigma)$ if $t(\sigma)\in\{S,S'\}$, and $\psi_{S,S'}(\sigma)=0$ otherwise.
\end{enumerate}
These labellings are indexed by a nonempty finite set $I$, since the colour set is finite; write $\lambda=(\lambda_k)_{k\in I}$ and $q=|I|$. They satisfy:
\begin{enumerate}[label=\textup{(\Alph*)}]
\item on every edge, the two labels in each coordinate are equal or at least one of them is zero;
\item every nonedge has distinct nonzero labels in some coordinate.
\end{enumerate}
For (A), let $\sigma\subsetneq\tau$ be an edge of $L$. If $\varphi_C(\sigma)$ and $\varphi_C(\tau)$ are nonzero, then $\sigma|_C$ and $\tau|_C$ are faces of $\tau$ with the same colours, hence equal. Since $t(\sigma)\subsetneq t(\tau)$, at most one of $t(\sigma),t(\tau)$ lies in $\{S,S'\}$, so one of $\psi_{S,S'}(\sigma),\psi_{S,S'}(\tau)$ is zero.

For (B), let $\{\sigma,\tau\}$ be a nonedge. If $t(\sigma)\subseteq t(\tau)$ then $\varphi_{t(\sigma)}$ gives $\sigma$ and $\tau$ the nonzero labels $\iota(\sigma)$ and $\iota(\tau|_{t(\sigma)})$, which are distinct since $\sigma\not\subseteq\tau$; the case $t(\tau)\subseteq t(\sigma)$ is symmetric. Otherwise $\psi_{t(\sigma),t(\tau)}$ gives them the distinct nonzero labels $\iota(\sigma)$ and $\iota(\tau)$.

Put $V_\sigma=\prod_{k\in I}D_{\lambda_k(\sigma)}$ and let $Y=\bigcup_\sigma V_\sigma$ be the union over the countably many vertices $\sigma$ of $L$, an open subset of $(\C^2)^I=\C^{2q}$. For a nonempty finite set $J$ of vertices,
\[
\bigcap_{\sigma\in J}V_\sigma
=\prod_{k\in I}\bigcap_{\sigma\in J}D_{\lambda_k(\sigma)}.
\]
If $J$ is not a simplex of $L$, then it contains a nonedge since $L$ is flag, and (B) makes some factor empty. If $J$ is a simplex, then (A) forces all nonzero labels in each coordinate to be equal, so every factor is $\C^2$ or a single $D_j$. For $\sigma\in J$, the factor in the coordinate of $\varphi_{t(\sigma)}$ is $D_{\iota(\sigma)}$. Hence the intersection is a proper subdomain of $\C^{2q}$ biholomorphic to $\C^{2q}$, that is, a Fatou--Bieberbach domain.

Thus the open cover $(V_\sigma)$ of $Y$ has nerve $L$ and contractible nonempty finite intersections, and the nerve lemma gives $Y\simeq|L|\simeq T$ \cite[Cor.~4G.3]{Hat02}. For each $y\in Y$, choose a cover member $V_\sigma$ containing $y$ and a biholomorphism $f:\C^{2q}\to V_\sigma$ with $f(0)=y$. Composing $f$ with the inclusion $V_\sigma\hookrightarrow Y$ gives a holomorphic map whose differential at $0$ is invertible. Thus $Y$ is strongly dominable and is Oka by Theorem~\ref{thm:strong-dom}.
\end{proof}

The dimension of this realization depends only on $n$: there are $a=2^{n+1}-1$ labellings $\varphi_C$ and at most $a(a-1)/2$ labellings $\psi_{S,S'}$, so $2q\le a(a+1)<4^{n+1}$.

If $T$ is a finite CW complex, $P_0$ can be chosen finite \cite[Thm.~2C.5]{Hat02}. Then the construction gives a finite good open cover of $Y$ all of whose nonempty finite intersections are Fatou--Bieberbach domains.

\begin{remark}[Stein manifolds]\label{rem:stein-types}
The homotopy types with Oka realizations are also exactly those of smooth manifolds \cite[Prop.~3]{KS10}, and of Stein manifolds. More precisely, every nonempty countable CW complex of dimension $d$ is homotopy equivalent to a Stein manifold of complex dimension $d$. This follows from the results of Eliashberg \cite[Thm.~1.3.1]{Eli90} for $d\ge3$ and Gompf \cite[Thm.~3.1]{Gom98} for $d=2$; the cases $d\le1$ are elementary. See \cite[Thm.~3.6]{BH26} for this formulation and its proof.

The dimension bound $d$ is sharp in general. A Stein manifold of complex dimension $n$ has the homotopy type of a CW complex of dimension at most $n$ \cite[Cor.~3.12.2]{For17}, so a Stein realization of $S^d$ must have complex dimension at least $d$. For the problem of Stein Oka realizations, see Problem~\ref{prob:stein-oka}.
\end{remark}

\subsection{Low-dimensional Oka open sets}\label{tw:domain-section}

For an open set $U\subset\C^{2q}$, we construct an Oka open set $D_U\subset\C^{2q+1}$ homotopy equivalent to $U$.

\begin{theorem}\label{thm:low-dimension}
Let $X$ be a nonempty countable CW complex of finite dimension $d$.
There is an Oka open set $D\subset\C^N$ homotopy equivalent to $X$, where
\[
 N=2\left\lfloor\frac d2\right\rfloor+3
 =\begin{cases}
 d+3,&d\text{ is even},\\
 d+2,&d\text{ is odd}.
 \end{cases}
\]
\end{theorem}

For $q\ge1$, equip $\C^{2q}=\C^q\oplus\C^q$ with the antilinear isometry
\[
 J(a,b)=(\overline b,-\overline a).
\]
It satisfies $J^2=-I$ and $\langle u,Ju\rangle=0$. Thus
\begin{equation}\label{tw:norm}
 \|u+\zeta Ju\|^2=(1+|\zeta|^2)\|u\|^2,
 \qquad u\in\C^{2q},\quad \zeta\in\C.
\end{equation}
The real-linear map $T_\zeta:u\mapsto u+\zeta Ju$ is invertible, with
\begin{equation}\label{tw:inverse}
 T_\zeta^{-1}(w)=\frac{w-\zeta Jw}{1+|\zeta|^2}.
\end{equation}
For a nonempty open set $U\subset\C^{2q}$, put
\begin{equation}\label{tw:domain}
 D_U=\left\{(w,\zeta)\in\C^{2q}\times\C:
       \frac{w-\zeta Jw}{1+|\zeta|^2}\in U\right\}.
\end{equation}
The smooth map
\begin{equation}\label{tw:diffeomorphism}
 \Phi:U\times\C\longrightarrow D_U,
 \qquad (u,\zeta)\longmapsto(u+\zeta Ju,\zeta)
\end{equation}
is a diffeomorphism by \eqref{tw:inverse}. Consequently,
$D_U\simeq U$.

\begin{lemma}\label{tw:uniform}
Suppose that $B(u_0,r)\subset U$. If $A,B\in\C^{2q}$ satisfy
\[
 \|A-u_0\|^2+\|B-Ju_0\|^2<r^2,
\]
then $(A+\zeta B,\zeta)\in D_U$ for every $\zeta\in\C$.
\end{lemma}
\begin{proof}
Write $\delta A=A-u_0$ and $\delta B=B-Ju_0$. By
\eqref{tw:norm}, \eqref{tw:inverse} and the Cauchy--Schwarz inequality,
\[
 \left\|T_\zeta^{-1}(A+\zeta B)-u_0\right\|
 =\frac{\|\delta A+\zeta\delta B\|}{\sqrt{1+|\zeta|^2}}
 \le\sqrt{\|\delta A\|^2+\|\delta B\|^2}<r
\]
for every $\zeta\in\C$.
\end{proof}

\begin{proposition}\label{tw:oka}
For every nonempty open set $U\subset\C^{2q}$, the open set $D_U$ is Oka.
\end{proposition}
\begin{proof}
Fix $y_0=(w_0,\zeta_0)\in D_U$, put
$u_0=T_{\zeta_0}^{-1}(w_0)$, and choose $r>0$ such that
$B(u_0,r)\subset U$. For the standard basis
$e_1,\ldots,e_{2q}$ of $\C^{2q}$, choose $\epsilon>0$ such that
\[
 \epsilon\sqrt{1+|\zeta_0|^2}<r/2,
\]
and put
\[
 B_0=Ju_0,\qquad B_j=Ju_0+\epsilon e_j\quad(1\le j\le2q).
\]
For each $j$, the pair $(w_0-\zeta_0B_j,B_j)$ satisfies the strict
inequality in Lemma~\ref{tw:uniform}. Since there are finitely many $j$, there is an open neighbourhood
$V\subset D_U$ of $y_0$ such that
\[
 \|w-\zeta B_j-u_0\|^2+\|B_j-Ju_0\|^2<r^2
 \qquad ((w,\zeta)\in V,\ 0\le j\le2q).
\]
Consider
\[
 h_j:V\times\C\longrightarrow\C^{2q+1},
 \qquad h_j((w,\zeta),t)=(w+tB_j,\zeta+t),
 \qquad 0\le j\le2q.
\]
For each $(w,\zeta)\in V$, the image of
$t\mapsto h_j((w,\zeta),t)$ is the complex affine line
$\{(w-\zeta B_j+\eta B_j,\eta):\eta\in\C\}$.
Lemma~\ref{tw:uniform} therefore gives $h_j(V\times\C)\subset D_U$.
Each $h_j$ is holomorphic on $V\times\C$ and satisfies $h_j(y,0)=y$. For every $y\in V$, $\partial_t h_j(y,0)=(B_j,1)$; subtracting $(B_0,1)$ from
the other vectors gives $(\epsilon e_j,0)$. Hence these $2q+1$ vectors
form a basis of the tangent space at every point of $V$.
Theorem~\ref{thm:entire-families}, applied componentwise, proves the
Oka property.
\end{proof}

\begin{proof}[Proof of Theorem~\ref{thm:low-dimension}]
By Lemma~\ref{lem:topological-open}, there is a nonempty open set
$V\subset\R^{2d+1}$ with $V\simeq X$. Set
\[
 q=\left\lfloor\frac d2\right\rfloor+1,
 \qquad
 U=V\times\R^{4q-(2d+1)}\subset\R^{4q}=\C^{2q}.
\]
Here $q\ge1$ and $4q-(2d+1)\ge0$.
The open set $U$ is homotopy equivalent to $V$.
Proposition~\ref{tw:oka} gives an Oka open set
$D_U\subset\C^{2q+1}$, and \eqref{tw:diffeomorphism} gives
$D_U\simeq U\simeq X$. The dimension is the displayed value of $N$.
\end{proof}

\begin{problem}[Dimension of Oka realizations]\label{prob:oka-dimension}
For each $d\ge1$, what is the least integer $n$ such that every
nonempty countable CW complex of dimension at most $d$ is homotopy
equivalent to an Oka open subset of $\C^n$?
\end{problem}

Theorem~\ref{thm:low-dimension}, applied in dimension $e\le d$ and
followed by taking a product with $\C^{N(d)-N(e)}$, gives the upper
bound $N(d)$ for this least integer.

\subsection{Extension of holomorphic functions}

\begin{proposition}\label{tw:envelope}
If $U\subset\C^{2q}$ is nonempty and connected, every holomorphic
function on $D_U$ extends uniquely to an entire function on
$\C^{2q+1}$. If $U$ is a proper open subset, then $D_U$ is not Stein.
\end{proposition}
\begin{proof}
Choose $B(u_0,r)\subset U$. The complex affine coordinate change
\[
 (w,\zeta)\longmapsto(w-u_0-\zeta Ju_0,\zeta)
\]
identifies $D_{B(u_0,r)}$ with
\[
 H_r=\{(v,\zeta)\in\C^{2q}\times\C:
             \|v\|<r\sqrt{1+|\zeta|^2}\}.
\]
Let $f$ be holomorphic on $H_r$. For $R>0$, Cauchy integration over
$|\eta|=R$ gives a holomorphic function
\[
 F_R(v,\zeta)=\frac{1}{2\pi i}\int_{|\eta|=R}
                 \frac{f(v,\eta)}{\eta-\zeta}\,d\eta
\]
on
\[
 \{\|v\|<r\sqrt{1+R^2}\}\times\{|\zeta|<R\}.
\]
For $\|v\|<r$, the entire vertical line lies in $H_r$, so the Cauchy
formula gives $F_R=f$ when also $|\zeta|<R$.
The identity theorem makes these extensions agree for different $R$.
Their domains exhaust $\C^{2q+1}$, giving an entire extension of $f$.

Apply this argument to the restriction of a holomorphic function on
$D_U$ to $D_{B(u_0,r)}$, using the affine coordinate change above.
The resulting entire function agrees with the original function on all
of $D_U$ by connectedness. The identity theorem also gives uniqueness.
If $U$ is proper, \eqref{tw:domain} shows that $D_U$ is a proper domain.
Its extension property prevents it from being a domain of holomorphy,
so it is not Stein.
\end{proof}

\section{Good open covers with prescribed nerves}\label{sec:prescribed-nerves}

\begin{theorem}\label{thm:prescribed-nerve}
Let $L$ be a finite simplicial complex with $v\ge1$ vertices and
$s$ maximal simplices. There is an Oka manifold $M_L$ of complex
dimension $D=5s+2$ with a good open cover $W_1,\ldots,W_v$ such that
each $W_i$ is biholomorphic to $\C^D$ and the nerve of the cover is
exactly $L$. In particular, $M_L\simeq|L|$.
The manifold $M_L$ has holomorphically trivial canonical bundle and
topologically trivial complex tangent bundle.
\end{theorem}

Lemma~\ref{nerve:convex} gives convex open sets $U_i\subset\C^s$
with nerve $L$. Choose proper holomorphic embeddings
$h_i:U_i\hookrightarrow\C^q$, where $q=2s+1$, and distinct constants
$\lambda_i\in\C$, and put $H_i=(h_i,\lambda_i h_i)$.
The transition maps on $(U_i\cap U_j)\times\C^{2q}$ are $(z,w)\mapsto(z,w+H_i(z)-H_j(z))$.
The formulas in \eqref{nerve:recovery} show that $H_i-H_j$ is proper
on $U_i\cap U_j$, which implies that the quotient is Hausdorff.

\subsection{A Hausdorff gluing construction}

\begin{proposition}\label{nerve:gluing}
Let $n\ge1$ and let $\{U_i\}_{i\in J}$ be a nonempty finite or countable
family of nonempty convex open subsets of $\C^n$. Put $D=5n+2$.
There is an Oka manifold $M$ with a good open cover
$\{W_i\}_{i\in J}$ having the same nerve as $\{U_i\}_{i\in J}$.
With $\C^D=\C^n\times\C^{4n+2}$, there are biholomorphisms
$\theta_i:\C^D\to W_i$ such that, for every finite $I\subseteq J$
with $|I|\ge2$ and every $i\in I$,
\[
 \theta_i^{-1}\left(\bigcap_{j\in I}W_j\right)
 =U_I\times\C^{4n+2},\qquad U_I=\bigcap_{j\in I}U_j.
\]
\end{proposition}
\begin{proof}
Each $U_i$ is Stein. By Narasimhan's embedding theorem
(see \cite[\S3.1, p.~781]{For13}), choose a proper holomorphic embedding
\[
 h_i:U_i\hookrightarrow\C^q,\qquad q=2n+1.
\]
Choose pairwise distinct constants $\lambda_i\in\C$, and put
\[
 H_i(z)=(h_i(z),\lambda_i h_i(z))\in\C^{2q}.
\]
For $i\ne j$, let $F_{ij}=H_i-H_j$ on $U_i\cap U_j$.
If $F_{ij}(z)=(a,b)$, then
\begin{equation}\label{nerve:recovery}
 h_i(z)=\frac{b-\lambda_j a}{\lambda_i-\lambda_j},
 \qquad
 h_j(z)=\frac{b-\lambda_i a}{\lambda_i-\lambda_j}.
\end{equation}
For a compact set $K\subset\C^{2q}$, let $K_i,K_j\subset\C^q$ be its
images under the two linear maps in \eqref{nerve:recovery}.
Then $F_{ij}^{-1}(K)$ is a closed subset of
$h_i^{-1}(K_i)\cap h_j^{-1}(K_j)$, which is a compact subset of
$U_i\cap U_j$. Hence $F_{ij}$ is proper.
In particular, the graph of $F_{ij}$ is closed in
$\C^n\times\C^{2q}$. Indeed, the second coordinates of a convergent sequence in the graph,
together with their limit, form a compact set. Properness places
the first coordinates in its compact inverse image in $U_i\cap U_j$,
so the limit remains in the graph.

Take the disjoint union
\[
 V=\coprod_{i\in J}V_i,
 \qquad V_i=\C^n\times\C^{2q}\cong\C^D.
\]
Points in the same $V_i$ are equivalent only when equal. For $i\ne j$,
impose the relation
\begin{equation}\label{nerve:relation}
 (z,w)_i\sim(z',w')_j
 \quad\Longleftrightarrow\quad
 z=z'\in U_i\cap U_j,
 \quad w'=w+H_i(z)-H_j(z).
\end{equation}
Symmetry follows by reversing the difference. For transitivity,
consider points in $V_i,V_j,V_k$ with successive pairs equivalent.
If $i=j$ or $j=k$, transitivity is immediate. Otherwise their common
first coordinate belongs to $U_i\cap U_k$, and
$(H_i-H_j)+(H_j-H_k)=H_i-H_k$. If $k=i$, the differences cancel and the two points in
$V_i$ are equal. Thus \eqref{nerve:relation} is an equivalence relation.

Let $M=V/{\sim}$ and let $Q:V\to M$ be the quotient map. The transition
on $(U_i\cap U_j)\times\C^{2q}$ is the biholomorphism
\begin{equation}\label{nerve:transition}
 \Phi_{ij}(z,w)=(z,w+H_i(z)-H_j(z)).
\end{equation}
It has inverse $\Phi_{ji}$, and
$\Phi_{jk}\circ\Phi_{ij}=\Phi_{ik}$ on triple overlaps.
For an open set $O\subset V$, put $O_i=O\cap V_i$. The intersection of
its saturation with $V_j$ is
\[
 Q^{-1}(Q(O))\cap V_j
 =O_j\cup\bigcup_{i\ne j}
 \Phi_{ij}\bigl(O_i\cap((U_i\cap U_j)\times\C^{2q})\bigr),
\]
which is open. Hence $Q$ is open.

The equivalence relation is closed in $V\times V$. On $V_i\times V_i$
it is the diagonal. On $V_i\times V_j$ with $i\ne j$, it is given by
$z=z'$ and $(z,w'-w)\in\operatorname{Graph}(F_{ij})$, so it is closed
by the graph property above. The sets $V_i\times V_j$ form a disjoint open cover of $V\times V$, so the equivalence relation is closed in $V\times V$.
To separate two distinct classes, choose representatives $x,y\in V$.
Closedness of the relation gives neighborhoods $O\ni x$ and $O'\ni y$
such that no point of $O$ is equivalent to a point of $O'$. Then $Q(O)$ and $Q(O')$ are
disjoint open neighborhoods of the two classes. Thus $M$ is Hausdorff.

Each restriction $Q|_{V_i}$ is an open injection, hence a homeomorphism
onto the open set $W_i=Q(V_i)$. The transitions
\eqref{nerve:transition} give $M$ a complex manifold structure.
Images of the countable bases of the $V_i$ form a countable base for
$M$. Thus $M$ is second countable and paracompact.
Put $\theta_i=Q|_{V_i}$, identifying $V_i$ with $\C^D$.
For $|I|\ge2$ and $i\in I$, relation \eqref{nerve:relation} gives
$\theta_i^{-1}(\bigcap_{j\in I}W_j)=U_I\times\C^{2q}$.
Every nonempty $U_I$ is convex, so $(W_i)_{i\in J}$ is a good open
cover with the same nerve as $(U_i)_{i\in J}$.

For $y=Q((a)_i)\in W_i$, the map $f:\C^D\to M$,
$f(\zeta)=Q((a+\zeta)_i)$, satisfies $f(0)=y$ and has invertible
differential at $0$. Thus $M$ is strongly dominable and is Oka by
Theorem~\ref{thm:strong-dom}.
\end{proof}

\begin{proposition}\label{nerve:bundles}
Every manifold constructed in Proposition~\ref{nerve:gluing} has
holomorphically trivial canonical bundle and topologically trivial
complex tangent bundle.
\end{proposition}
\begin{proof}
Retain the notation of the construction in
Proposition~\ref{nerve:gluing}, with $q=2n+1$.
The differentials of the transitions \eqref{nerve:transition} are
\[
 D\Phi_{ij}=
 \begin{pmatrix}
 I_n&0\\
 dH_i-dH_j&I_{2q}
 \end{pmatrix}.
\]
Their determinants equal $1$, so the standard holomorphic volume forms
$dz_1\wedge\cdots\wedge dz_n\wedge dw_1\wedge\cdots\wedge dw_{2q}$
agree on overlaps and trivialize the canonical bundle.
Since the transition maps preserve $z$, the formula
$p(Q((z,w)_i))=z$ gives a holomorphic submersion $p:M\to\C^n$.
The vertical vector fields
$\partial/\partial w_1,\ldots,\partial/\partial w_{2q}$ also agree on overlaps.
There is therefore an exact sequence of holomorphic vector bundles
\[
 0\longrightarrow M\times\C^{2q}\longrightarrow TM
 \xrightarrow{\,dp\,}M\times\C^n\longrightarrow0.
\]
Choose a continuous Hermitian metric on $TM$.
The orthogonal complement of $\ker dp$ is a continuous complex
subbundle, and $dp$ restricts to an isomorphism from this complement
onto $M\times\C^n$. Thus the sequence splits as a sequence of
topological complex vector bundles.
Both outer bundles are trivial, hence
$TM\cong M\times\C^{n+2q}$ as a topological complex vector bundle.
\end{proof}

\subsection{Finite nerves}

\begin{lemma}\label{nerve:convex}
Let $L$ be a finite simplicial complex with $v\ge1$ vertices and
$s$ maximal simplices. There are bounded convex open sets
$U_1,\ldots,U_v\subset\C^s$ whose nerve is $L$.
\end{lemma}
\begin{proof}
We represent the vertices of $L$ by faces of a simplex and then replace each face by its open $\epsilon$-neighbourhood; compare \cite[Definition~3.1 and proofs of Theorems~3.3
and~3.14]{White23}. Label the vertices by $1,\ldots,v$ and the maximal
simplices by $\sigma_1,\ldots,\sigma_s$. With $e_1,\ldots,e_s$ the
standard basis of $\R^s$, put
\[
 C_i=\operatorname{conv}\{e_a:i\in\sigma_a\}.
\]
These are nonempty compact faces of the standard simplex. Uniqueness
of barycentric coordinates gives, for each nonempty vertex set $I$,
\[
 \bigcap_{i\in I}C_i
 =\operatorname{conv}\{e_a:I\subset\sigma_a\},
\]
where the convex hull of the empty set is empty. This intersection is
nonempty exactly when $I$ is a simplex of $L$.

Regard $\R^s$ as the standard real subspace of $\C^s$. For a nonempty vertex set $I$ that is not a simplex of $L$, the continuous function
\[
 z\longmapsto\max_{i\in I}\operatorname{dist}(z,C_i)
\]
tends to infinity as $|z|\to\infty$ and has a strictly positive minimum
$\delta_I$. Choose $\epsilon>0$ smaller than all these finitely many
minima; if there are no nonfaces, choose any $\epsilon>0$.
The sets
\[
 U_i=\{z\in\C^s:\operatorname{dist}(z,C_i)<\epsilon\}
\]
are bounded, convex and open. If $I$ is a simplex of $L$, then
$\bigcap_{i\in I}C_i\subset\bigcap_{i\in I}U_i$ is nonempty.
If $I$ is not a simplex, then $\bigcap_{i\in I}U_i=\varnothing$ by the
choice of $\epsilon$. Thus the nerve is $L$.
\end{proof}

\begin{proof}[Proof of Theorem~\ref{thm:prescribed-nerve}]
Apply Lemma~\ref{nerve:convex} and Proposition~\ref{nerve:gluing} with
$n=s$. The resulting good open cover has $v$ members and nerve exactly
$L$, and the manifold has complex dimension $5s+2$.
The nerve lemma~\cite[Corollary~4G.3]{Hat02} gives $M_L\simeq|L|$.
Proposition~\ref{nerve:bundles} gives the two bundle assertions.
\end{proof}

\subsection{Countable covers with a linear dimension bound}

\begin{proposition}\label{nerve:open-model}
Let $U\subset\C^n$ be a nonempty open set, where $n\ge1$.
There is an Oka manifold $M_U\simeq U$ of complex dimension $5n+2$
with a countable good open cover by open sets biholomorphic to $\C^{5n+2}$.
It has holomorphically trivial canonical bundle and topologically
trivial complex tangent bundle.
\end{proposition}
\begin{proof}
Choose a countable cover of $U$ by Euclidean open balls contained in
$U$. Every nonempty finite intersection is convex.
Proposition~\ref{nerve:gluing} gives a manifold with a good cover having
the same nerve as the ball cover. Both $U$ and $M_U$ are paracompact,
so the nerve lemma~\cite[Corollary~4G.3]{Hat02} gives $M_U\simeq U$.
Proposition~\ref{nerve:bundles} gives the bundle assertions.
\end{proof}

\begin{corollary}\label{nerve:countable}
Every nonempty countable CW complex $X$ of finite dimension $d$ has an
Oka realization $M$ of complex dimension $5d+7$ with a countable good
open cover by sets biholomorphic to $\C^{5d+7}$. The manifold $M$ has
holomorphically trivial canonical bundle and topologically trivial
complex tangent bundle.
\end{corollary}
\begin{proof}
Take $V\subset\R^{2d+1}$ with $V\simeq X$ from
Lemma~\ref{lem:topological-open}. Apply
Proposition~\ref{nerve:open-model} to
$U=V\times\R\subset\R^{2d+2}=\C^{d+1}$.
The resulting manifold has dimension $5(d+1)+2=5d+7$ and the stated cover and
bundle properties.
\end{proof}

\section{Homotopy types and holomorphic representatives}\label{sec:applications}

We write $[X,T]$ for homotopy classes of continuous maps $X\to T$. For $d\ge1$, a map is \emph{$d$-connected} if it induces a bijection on $\pi_0$ and, at every base point, isomorphisms on $\pi_i$ for $0<i<d$ and a surjection on $\pi_d$.

\subsection{Moore and Eilenberg--MacLane spaces}

By Theorem~\ref{intro:countable}, realizing a homotopy type reduces to finding a countable finite-dimensional CW model.

A Moore space of type $(G,n)$ is a connected CW complex with reduced
integral homology $G$ in degree $n$ and zero in all other degrees;
for $n\ge2$ it is also required to be simply connected. For $n\ge2$,
its homotopy type is determined by $G$ and $n$
\cite[Example~4.34]{Hat02}. For $n=1$, $M(G,1)$
denotes the CW complex constructed in the following proof, with its
chosen attaching maps.

\begin{corollary}[Moore spaces]\label{cor:all-moore}
For every countable abelian group $G$ and $n\ge1$, the Moore space $M(G,n)$ obtained from a countable free resolution of $G$ has an Oka realization.
\end{corollary}
\begin{proof}
Choose a free resolution
\[
0\longrightarrow F_1\xrightarrow{\partial}F_0\longrightarrow G\longrightarrow0
\]
by countable free abelian groups. Start with a wedge of $n$-spheres
indexed by a basis of $F_0$, and attach one $(n+1)$-cell for each basis
element of $F_1$, with attaching maps inducing $\partial$ on $H_n$
\cite[Example~2.40]{Hat02}. Cellular homology gives
\[
\widetilde H_n(M(G,n);\Z)=\operatorname{coker}\partial=G,
\qquad H_{n+1}(M(G,n);\Z)=\ker\partial=0,
\]
and all other reduced homology groups vanish. For $n\ge2$, the
construction is simply connected. The complex is countable and has
dimension at most $n+1$, so Theorem~\ref{intro:countable} applies.
\end{proof}

\begin{corollary}[aspherical Oka manifolds]\label{cor:aspherical}
For a group $G$, $K(G,1)$ has an Oka realization if and only if $G$ is countable with $\operatorname{cd}_\Z G<\infty$.
\end{corollary}
\begin{proof}
An Oka manifold has a countable finite-dimensional CW model by Theorem~\ref{intro:countable}, so an Oka realization of $K(G,1)$ forces $G$ to be countable with finite integral cohomological dimension. Conversely, if $\operatorname{cd}_\Z G<\infty$, then $K(G,1)$ has a finite-dimensional CW model by the Eilenberg--Ganea and Stallings--Swan theorems \cite[Thms.~F.3.2--F.3.3]{Dav08}. Let $K$ be such a finite-dimensional model. Choose a connected countable
subcomplex $A_0\subset K$ containing a base point and loops representing
all elements of $G$. Inductively, for every $i\ge1$, choose
null-homotopies in $K$ for representatives of the kernel of
$\pi_i(A_j)\to\pi_i(K)$, and choose a connected countable subcomplex $A_{j+1}\subset K$ containing
$A_j$ and the images of these null-homotopies. These choices are countable
because the homotopy groups of a countable CW complex are countable.
Then $A=\bigcup_j A_j$ is a countable subcomplex of $K$.
Every map from a sphere into $A$ has image in a finite subcomplex,
hence in some $A_j$, and every element of the kernel is killed in
$A_{j+1}$. The map $\pi_1(A)\to G$ is surjective by the choice of
$A_0$, and the higher homotopy groups of $K$ vanish. Thus
$A\hookrightarrow K$ induces isomorphisms on all homotopy groups.
Whitehead's theorem gives $A\simeq K$. Apply Theorem~\ref{intro:countable}.
\end{proof}

\begin{corollary}[higher Eilenberg--MacLane spaces]\label{cor:higher-em}
Let $G\ne0$ be an abelian group and $n\ge2$. The following are equivalent: $K(G,n)$ has an Oka realization; $K(G,n)$ is homotopy equivalent to a finite-dimensional CW complex; $G$ is a finite-dimensional rational vector space and $n$ is odd.
\end{corollary}
\begin{proof}
By Theorem~\ref{intro:countable}, an Oka realization gives a
countable finite-dimensional CW model. Suppose that $K(G,n)$ has a
finite-dimensional CW model. Miller's theorem \cite[Thm.~A]{Mil84}
implies that the space of pointed maps $B\Z/p\to K(G,n)$ is weakly
contractible for every prime $p$. Its homotopy groups at the constant
map are $\widetilde H^{n-i}(B\Z/p;G)$ for $1\le i\le n$, and its
set of path components is $\widetilde H^n(B\Z/p;G)$. In particular,
$H^1(B\Z/p;G)=H^2(B\Z/p;G)=0$. The universal coefficient theorem
gives
\[
\begin{aligned}
H^1(B\Z/p;G)&=\operatorname{Hom}_{\Z}(\Z/p,G)
 =\{g\in G:pg=0\},\\
H^2(B\Z/p;G)&=\operatorname{Ext}^1_{\Z}(\Z/p,G)=G/pG.
\end{aligned}
\]
Thus multiplication by every prime is an isomorphism on $G$, so $G$
is a rational vector space.

For a finite-dimensional rational subspace $W\subset G$, choose a
rational linear projection $G\to W$. The induced maps
\[
K(W,n)\longrightarrow K(G,n)\longrightarrow K(W,n)
\]
have composite homotopic to the identity, so $H^*(K(W,n);\Q)$ is a
direct summand of $H^*(K(G,n);\Q)$. If $\dim_{\Q}W=s$, the
calculation in \cite[Prop.~6.7.2]{MP12} and the K\"unneth formula
identify $H^*(K(W,n);\Q)$ with the polynomial algebra on $s$
degree-$n$ generators when $n$ is even, and with the exterior algebra
on $s$ degree-$n$ generators when $n$ is odd. In the even case, any
nonzero $W$ therefore gives nonzero cohomology in arbitrarily high
degrees. In the odd case, $H^{sn}(K(W,n);\Q)\ne0$, so an
infinite-dimensional $G$ has the same consequence by taking $s$
arbitrarily large. Finite-dimensional CW complexes have no cohomology
above their dimension. Hence $G$ is finite-dimensional and $n$ is odd.

Conversely, let $G=\Q^r$ and let $n\ge3$ be odd. The rationalization
of $S^n$ is $K(\Q,n)$ and is the mapping telescope of self-maps of
$S^n$ of degrees $2,3,4,\dots$ \cite[Prop.~6.6.4 and Cor.~6.7.3]{MP12}.
This telescope is a countable CW complex of dimension $n+1$.
Taking the product of $r$ such models gives a countable CW model of
$K(\Q^r,n)$ of dimension $r(n+1)$. Theorem~\ref{intro:countable}
gives an Oka realization.
\end{proof}

For example, $K(\Q,3)$ has a countable four-dimensional CW model given
by a mapping telescope and hence an Oka realization. Its third integral
homology group is $\Q$, so it has no finite CW model.

Realizing finite-dimensional skeleta gives highly connected maps from Oka manifolds.

\begin{corollary}[$d$-connected maps from Oka manifolds]\label{cor:finite-skeleta-oka}
For every $d\ge1$, every countable CW complex $T$ admits $d$-connected maps from Oka manifolds. In particular, $K(G,n)$ admits such maps for every countable abelian group $G$ and $n\ge1$, and so does $K(G,1)$ for every countable group $G$.
\end{corollary}
\begin{proof}
The inclusion $T^{(d)}\hookrightarrow T$ is $d$-connected by cellular approximation; realize $T^{(d)}$ by Theorem~\ref{intro:countable}. Eilenberg--MacLane spaces of countable groups have countable CW models \cite[\S4.2]{Hat02}.
\end{proof}

\subsection{Holomorphic representatives of cohomology classes}

The Oka principle converts the $d$-connected maps of Corollary~\ref{cor:finite-skeleta-oka} into statements about spaces of holomorphic maps from Stein manifolds.

\begin{proposition}[homotopy groups of mapping spaces]\label{prop:hol-rep}
Let $T$ be a countable CW complex and $m,r\ge0$. There are an Oka open set $Y\subset\C^N$ and a map $f:Y\to T$ such that, for every Stein manifold $X$ of complex dimension at most $m$, the natural map
\[
\mathcal O(X,Y)\longrightarrow\mathcal C(X,T),\qquad g\longmapsto f\circ g,
\]
is bijective on $\pi_0$ and induces isomorphisms on $\pi_i$ for $1\le i\le r$ at every base point $g\in\mathcal O(X,Y)$. If $T$ is finite-dimensional, $Y$ and $f$ can be chosen so that the displayed map is a weak homotopy equivalence for every Stein manifold $X$.
\end{proposition}
\begin{proof}
Choose an integer $d>m+r$, realize $T^{(d)}$ by
Theorem~\ref{intro:countable}, and let $f:Y\to T^{(d)}\subset T$
be the resulting map. The inclusion of the $d$-skeleton is
$d$-connected. Every Stein manifold $X$ of complex dimension at most
$m$ has a CW model of dimension at most $m$
\cite[Cor.~3.12.2]{For17}. A $d$-connected map induces a
$(d-m)$-connected map on mapping spaces from such a CW complex.
After replacing the map by its mapping-path fibration,
obstruction theory gives lifts and relative lifts for the parameter
pairs $(D^{i+1},S^i)$ whenever the corresponding relative CW dimension
$m+i+1$ is at most $d$. Since $d-m>r$, the map
$\mathcal C(X,Y)\to\mathcal C(X,T)$ is bijective on $\pi_0$ and
induces isomorphisms on $\pi_i$ for $1\le i\le r$, at every base
point. The Oka principle makes
$\mathcal O(X,Y)\hookrightarrow\mathcal C(X,Y)$ a weak homotopy
equivalence \cite[Cor.~5.5.6]{For17}. If $T$ is finite-dimensional,
realize all of $T$ and take $f$ to be a homotopy equivalence; then
$\mathcal C(X,Y)\to\mathcal C(X,T)$ is a homotopy equivalence for
every $X$.
\end{proof}

Apply Proposition~\ref{prop:hol-rep} to the infinite loop space representing $E^k$.

\begin{corollary}[holomorphic representatives for cohomology theories]\label{cor:theories}
Let $E$ be a spectrum with countable homotopy groups, and let $k\in\Z$ and $m\ge0$. There are an Oka open set $Y\subset\C^N$ and a class $u\in E^k(Y)$ such that, for every Stein manifold $X$ of complex dimension at most $m$, pullback gives a bijection
\[
\pi_0\mathcal O(X,Y)\longrightarrow E^k(X),\qquad [g]\longmapsto g^*u.
\]
The bijections are natural with respect to holomorphic maps between Stein sources. There exist an Oka manifold $Y$ and a class $u\in E^k(Y)$ giving these bijections for every Stein manifold $X$ if and only if $E_k=\Omega^\infty\Sigma^kE$ has an Oka realization.
\end{corollary}
\begin{proof}
Take $E_k=\Omega^\infty\Sigma^k E$, so that
$E^k(X)=[X_+,E_k]_*=[X,E_k]$, where $X_+$ is $X$ with a disjoint
base point and the subscript $*$ denotes based homotopy classes.
The space $E_k$ has countably many components and countable homotopy
groups at every base point, since $\pi_i(E_k)=\pi_{i-k}(E)$ and all
components are homotopy equivalent. The CW approximation construction
\cite[Prop.~4.13]{Hat02} gives a countable CW model: in each dimension
one attaches countably many cells to represent the target homotopy
group and then kill the kernel, using countability of the homotopy
groups of the CW complex already constructed. Carry out this
construction on each component and take their countable disjoint
union, which we again denote by $E_k$. Apply
Proposition~\ref{prop:hol-rep} to this model, and let $u\in E^k(Y)$
be the class of $f$.
If $E_k$ has an Oka realization,
realize it by Theorem~\ref{intro:countable}, take $f$ to be a homotopy
equivalence, and apply the Oka principle for every Stein source.

Conversely, suppose that an Oka manifold $Y$ and a class
$u\in E^k(Y)$ give the displayed bijections for every Stein manifold,
and let $f:Y\to E_k$ represent $u$. The Oka principle makes
$f_*:[X,Y]\to[X,E_k]$ bijective for every Stein $X$. Taking $X$ to
be a point gives a bijection on components. For each $j\ge1$, the
affine quadric
\[
Q_j=\{z_0^2+\cdots+z_j^2=1\}\subset\C^{j+1}
\]
is Stein. Writing $z=x+iy$ identifies $Q_j$ diffeomorphically with
$TS^j$, so $Q_j\simeq S^j$. Thus $f_*:[S^j,Y]\to[S^j,E_k]$ is
bijective. Since $f$ is bijective on components, these bijections
restrict to each pair of corresponding components.

Fix $y\in Y$, and write $Y_0$ and $T_0$ for the components containing
$y$ and $f(y)$. Every component of the grouplike $H$-space $E_k$ is
homotopy equivalent by translation to its identity component. Thus
$T_0$ is simple: its fundamental group is abelian and acts trivially
on $\pi_j(T_0,f(y))$ for every $j\ge2$
\cite[Example~4A.3]{Hat02}. For a connected space, free homotopy
classes of maps from spheres are the orbits of the fundamental group
action on the based homotopy group, with conjugation in degree one
\cite[Prop.~4A.2]{Hat02}.
Hence $f_*:\pi_1(Y_0,y)\to\pi_1(T_0,f(y))$ gives a bijection on
conjugacy classes. Since the target is abelian, this homomorphism is
surjective. An element in its kernel has the same image conjugacy
class as the identity, so injectivity on conjugacy classes forces the
element to be the identity.

For $j\ge2$, the induced bijection is
\[
\pi_j(Y_0,y)/\pi_1(Y_0,y)\longrightarrow\pi_j(T_0,f(y)).
\]
It implies surjectivity of $f_*$ on $\pi_j$. If an element lies in
the kernel, its orbit has the same image as the orbit of $0$.
Injectivity of the displayed map places it in that orbit, which is
$\{0\}$ because the fundamental group acts by automorphisms. Thus
$f_*$ is also injective. Since $y$ was arbitrary, $f$ is a weak
homotopy equivalence. The manifold $Y$ has CW type, so Whitehead's
theorem gives $Y\simeq E_k$.
\end{proof}

For a countable abelian group $G$ and ordinary cohomology $H^n(\,\cdot\,;G)$ with $n\ge1$, the target in
Corollary~\ref{cor:theories} can be chosen to be a domain. For complex
$K$-theory in degree zero, the target has $\pi_0(Y)\cong\Z$.

\begin{remark}[representability for all Stein manifolds]\label{rem:bound-needed}
Let $G\ne0$ be a countable abelian group and $n\ge1$. By
Corollary~\ref{cor:theories}, there exist an Oka manifold $Y$ and a
class $u\in H^n(Y;G)$ inducing bijections
$\pi_0\mathcal O(X,Y)\to H^n(X;G)$ for every Stein manifold $X$ if
and only if $K(G,n)$ has an Oka realization. By
Corollaries~\ref{cor:aspherical} and~\ref{cor:higher-em}, this holds if
and only if $\operatorname{cd}_\Z G<\infty$ when $n=1$, or $G$ is a
finite-dimensional rational vector space and $n$ is odd when $n\ge2$.
\end{remark}

\section{Fundamental groups and algebraic realizations}\label{sec:groups}

We determine the fundamental groups of Oka manifolds, algebraically elliptic manifolds, and compact K\"ahler Oka manifolds.

\subsection{Countable fundamental groups in dimension five}

\begin{corollary}\label{cor:countable-groups}
The fundamental groups of connected Oka manifolds are exactly the
countable groups. Every countable group occurs as the fundamental
group of an Oka domain $D\subset\C^5$ containing a properly
holomorphically embedded Stein surface whose inclusion is a homotopy
equivalence.
\end{corollary}
\begin{proof}
A second-countable manifold has the homotopy type of a countable CW complex, hence
its fundamental group is generated by countably many elements and is
countable.
Conversely, let $G$ be a countable group. Taylor's construction gives
a connected Stein surface $S$ with $\pi_1(S)\cong G$
\cite[Theorem~3.5 and Example~3.6]{Taylor97}. It first replaces the
countable presentation of $G$ by a locally finite two-dimensional
presentation complex and realizes its homotopy type by a Stein surface,
using Gompf's theorem~\cite{Gom98}.
The embedding theorem of Eliashberg and Gromov~\cite{EG92} gives a
proper holomorphic embedding $S\hookrightarrow\C^4$: for a Stein
manifold of dimension $n$, their bound is the least integer strictly
larger than $(3n+1)/2$, which is four for $n=2$.
Theorem~\ref{thm:stein-open-embedding} with $q=2$ gives an Oka domain
$D\subset\C^5$ and a proper holomorphic embedding $S\hookrightarrow D$
that is a homotopy equivalence.
\end{proof}

For a finitely presented group, replace its finite presentation
complex by a homotopy-equivalent finite simplicial complex
\cite[Thm.~2C.5]{Hat02}. Theorem~\ref{thm:prescribed-nerve} then
gives an Oka realization with a finite good open cover whose members
are biholomorphic to $\C^N$.

\subsection{Algebraically elliptic manifolds}\label{sec:algebraic-groups}

All varieties in this section are over $\C$, and fundamental groups are those of the underlying complex spaces. An algebraic manifold is a smooth connected algebraic variety. It is \emph{algebraically elliptic} if there is an algebraic vector bundle $E\to Y$ and a morphism $s:E\to Y$ such that $s(0_y)=y$ and $(ds)_{0_y}(E_y)=T_{Y,y}$ for every $y\in Y$. Such a map is a \emph{dominating algebraic spray}. Algebraic ellipticity is equivalent to algebraic subellipticity by \cite[Thm.~0.1]{KZ23}. The analytification of an algebraically elliptic manifold is Oka \cite[Cor.~5.6.14]{For17}.

The finiteness theorem and configuration-space constructions below originate in \cite{Kus22}, where they are stated in terms of algebraic subellipticity.

\begin{proposition}[{\cite[Thm.~3.1]{Kus22}}]\label{prop:serre-affine-finite}
If an algebraic manifold $Y$ admits a dominant morphism $\C^m\to Y$, then $\pi_1(Y)$ is finite. In particular, every algebraically elliptic variety has finite fundamental group.
\end{proposition}

\begin{proof}
Following Serre's argument \cite{Ser59}, let $n=\dim Y$.
Choose a point where the differential has rank $n$, and restrict the
morphism to an affine $n$-plane transverse to the kernel there. This
gives a dominant, generically finite morphism $f:\C^n\to Y$.
There are nonempty Zariski open sets $U\subset\C^n$ and $V\subset Y$
such that $f:U\to V$ is a finite \'etale morphism of some degree $D$.
Its analytification is a $D$-sheeted covering, so
$H=f_*\pi_1(U)$ has index $D$ in $\Gamma=\pi_1(V)$. Smoothness of
$Y$ makes $\Gamma\to\pi_1(Y)$ surjective. This surjection kills $H$,
since $\pi_1(U)\to\pi_1(Y)$ factors through $\pi_1(\C^n)=1$.
Hence $|\pi_1(Y)|\le[\Gamma:H]=D$. For a dominating spray
$s:E\to Y$, the restriction $s|_{E_y}:E_y\cong\C^m\to Y$ is a
submersion at $0_y$, hence dominant.
\end{proof}

\begin{corollary}[universal covers]\label{grp:universal-cover}
The universal cover of an algebraically elliptic manifold is a finite \'etale algebraic cover and is algebraically elliptic \cite[Cor.~1.5]{Kus22}, \cite[Cor.~0.3]{KZ23}.
\end{corollary}
\begin{proof}
Proposition~\ref{prop:serre-affine-finite} makes the universal cover finite, and the Riemann existence theorem gives it an algebraic structure for which the covering is \'etale. Algebraic subellipticity lifts through finite \'etale covers \cite[Prop.~6.4.10]{For17}. Apply the equivalence with algebraic ellipticity \cite[Thm.~0.1]{KZ23}.
\end{proof}

A reduced irreducible variety $Y$ is \emph{flexible} if, for every
$y\in Y_{\mathrm{sm}}$, the vectors
\[
 \left.\frac{d}{dt}\right|_{t=0}\varphi(t,y),
 \qquad \varphi:\mathbb G_a\times Y\to Y,
\]
span $T_{Y,y}$, where $\mathbb G_a=(\C,+)$ and $\varphi$ ranges over
algebraic $\mathbb G_a$-actions. The \emph{special automorphism group}
$\operatorname{SAut}(Y)$ is the subgroup of $\operatorname{Aut}(Y)$
generated by these one-parameter subgroups.
By quasi-compactness of $Y_{\mathrm{sm}}$, finitely many such actions
$\varphi_1,\dots,\varphi_k$ span the tangent space at every point of
$Y_{\mathrm{sm}}$. Their composition
\[
(y,t)\longmapsto\varphi_1(t_1,\varphi_2(t_2,\dots,\varphi_k(t_k,y)\dots))
\]
defines a dominating algebraic spray on $Y_{\mathrm{sm}}$
with underlying vector bundle
$Y_{\mathrm{sm}}\times\C^k\to Y_{\mathrm{sm}}$. Thus $Y_{\mathrm{sm}}$ is algebraically elliptic. The \emph{ordered configuration space} is $\Conf_n(Y)=\{(y_1,\dots,y_n)\in Y^n:y_i\ne y_j\text{ for }i\ne j\}$, for $n\ge1$.

\begin{proposition}[flexible configuration spaces]\label{prop:flexible-config}
If $Y$ is a flexible quasi-affine algebraic manifold of dimension at least two, then $\Conf_n(Y)$ and its quotients $\Conf_n(Y)/H$, $H\le\mathfrak S_n$, are flexible quasi-affine algebraic manifolds.
\end{proposition}

\begin{proof}
For quasi-affine varieties, flexibility is equivalent to transitivity
of $\operatorname{SAut}$. In dimension at least two,
$\operatorname{SAut}(Y)$ acts infinitely transitively on $Y$
\cite[Thm.~2]{APS14}, \cite[Thm.~2.12]{FKZ16}. Hence its diagonal
action is transitive on $\Conf_n(Y)$. The diagonal
$\mathbb G_a$-actions commute with permutations, so they descend to
$\Conf_n(Y)/H$ and generate a group acting transitively there.

The permutation action is free, so its quotients are smooth. Choose
an affine variety $\overline Y$ containing $Y$ as an open subset.
There is an open embedding
\[
\Conf_n(Y)/H\hookrightarrow\overline Y^{\,n}/H,
\]
whose target is affine. Thus both configuration spaces are quasi-affine,
and the transitivity criterion gives flexibility of each.
\end{proof}

\begin{remark}\label{grp:singular-configurations}
The same argument applies to a flexible quasi-affine variety $Y$ with singularities. The smooth locus of $\Conf_n(Y)$ is $\Conf_n(Y_{\mathrm{sm}})$; the diagonal additive-group actions are transitive there. The permutation quotient is \'etale, so its smooth locus is the quotient of this smooth locus. Hence $\Conf_n(Y)$ and $\Conf_n(Y)/H$ are flexible varieties. The ordered and unordered configuration-space conclusions of \cite[Props.~2.6 and~3.4]{Kus22} follow by taking $H=1$ and $H=\mathfrak S_n$.
\end{remark}

For $r\ge2$, the space $\Conf_n(\C^r)$ is $(2r-2)$-connected.
Indeed, it is the complement of finitely many diagonals of real
codimension $2r$ in $(\C^r)^n$. Relative transversality makes a
filling $(i+1)$-disc transverse to every diagonal; it then avoids
them when $i+1<2r$. For $r=2$, the construction in
\cite[Thm.~3.3]{Kus22} realizes all finite groups as fundamental groups.

\begin{theorem}[finite groups via configuration spaces]\label{thm:elliptic-finite-groups}
Let $G\le\mathfrak S_n$ be a finite group. Then $Y_G=\Conf_n(\C^2)/G$ is a flexible quasi-affine algebraic manifold of dimension $2n$ with $\pi_1(Y_G)\cong G$. In particular, every finite group is the fundamental group of a flexible quasi-affine variety of dimension at most $2|G|$, and the fundamental groups of algebraically elliptic varieties are exactly the finite groups.
\end{theorem}

\begin{proof}
Proposition~\ref{prop:flexible-config} gives flexibility. The preceding connectivity calculation shows that $\Conf_n(\C^2)$ is connected and simply connected. A subgroup of $\mathfrak S_n$ acts freely on ordered configurations, so $\Conf_n(\C^2)$ is the universal cover of $Y_G$. For the second statement, use the left regular permutation action $G\le\mathfrak S_{|G|}$, and use Proposition~\ref{prop:serre-affine-finite} for the converse.
\end{proof}

For a nonconstant regular function $f$ on an affine variety $Z$, the
\emph{affine suspension} associated with $f$ is the hypersurface
$\{(u,v,z)\in\C^2\times Z:uv=f(z)\}$.

\begin{proposition}[cyclic groups on surfaces]\label{grp:cyclic-surfaces}
For every integer $m\ge1$, there is a smooth flexible quasi-affine surface with fundamental group $\Z/m\Z$.
\end{proposition}
\begin{proof}
For $m=1$, take $\C^2$. For $m>1$, use the action of $\mu_m=\{\zeta\in\C:\zeta^m=1\}$ on $\C^2$ given by
\[
\zeta\cdot(z,w)=(\zeta z,\zeta^{-1}w).
\]
It is free off the origin, and its invariant ring gives
\[
\C^2/\mu_m
 =\operatorname{Spec}\C[z^m,w^m,zw]
 \cong\{(u,v,t)\in\C^3:uv=t^m\}.
\]
This affine surface is a suspension over $\C$ and is flexible by \cite[Thm.~3.1]{AZK12}. Its smooth locus is $Y_m=(\C^2\setminus\{0\})/\mu_m$, which is therefore flexible. Since $\C^2\setminus\{0\}$ is simply connected, the quotient map is the universal cover of $Y_m$ and $\pi_1(Y_m)=\mu_m\cong\Z/m\Z$. This is the two-dimensional construction of \cite[Prop.~3.5]{Kus22}.
\end{proof}

\begin{corollary}[sphere homotopy types]\label{grp:spheres}
For a positive integer $n$, the sphere $S^n$ is homotopy equivalent to an algebraically elliptic manifold if and only if $n>1$.
\end{corollary}
\begin{proof}
Proposition~\ref{prop:serre-affine-finite} excludes $n=1$. For $n>1$, the affine quadric
\[
\Sigma^n=\{z\in\C^{n+1}:z_0^2+\cdots+z_n^2=1\}
\]
has equation
\[
 uv=1-z_2^2-\cdots-z_n^2
\]
after the linear substitution $u=z_0+iz_1$, $v=z_0-iz_1$.
Thus it is an affine suspension over $\C^{n-1}$. It is flexible by \cite[Thm.~3.2]{AZK12}, and hence algebraically elliptic. Writing $z=x+iy$ identifies $\Sigma^n$ with $TS^n$ by
\[
(a,y)\longmapsto\sqrt{1+|y|^2}\,a+iy,
\qquad |a|=1,\quad a\cdot y=0.
\]
The zero section is a deformation retract, so $\Sigma^n\simeq S^n$, as in \cite[Cor.~1.4]{Kus22}.
\end{proof}

\begin{corollary}[algebraic approximation after power maps]\label{grp:power-approximation}
Let $Y$ be an algebraically elliptic manifold, and let $e=\operatorname{lcm}\{\operatorname{ord}(g):g\in\pi_1(Y)\}$ be the exponent of its finite fundamental group. For every holomorphic map $f:\C^*\to Y$ and every positive integer $n$ divisible by $e$, the map $z\mapsto f(z^n)$ is a locally uniform limit of algebraic morphisms $\C^*\to Y$.
\end{corollary}
\begin{proof}
Let $q_n:\C^*\to\C^*$, $q_n(z)=z^n$. The homomorphism $(f\circ q_n)_*$ sends a generator of $\pi_1(\C^*)\cong\Z$ to $f_*(1)^n=1$, since $e\mid n$. Since $\C^*\simeq S^1$, the map $z\mapsto f(z^n)$ is null-homotopic. On every compact $\mathcal O(\C^*)$-convex subset, the algebraic approximation theorem \cite[Cor.~3.2]{For06Flex} therefore approximates it by algebraic morphisms defined on all of $\C^*$. An exhaustion by such compact subsets gives the claimed convergence. Compare \cite[Cor.~1.6]{Kus22}.
\end{proof}

\begin{remark}[proper birational maps]\label{grp:birational-dominability}
Let $b:Y_2\to Y_1$ be a proper birational morphism between algebraic
manifolds. There is a dominant morphism $\C^m\to Y_1$ for some $m$ if
and only if there is one to $Y_2$ \cite[Rem.~3.2]{Kus22}.
Composing a dominant morphism to $Y_2$ with $b$ gives one to $Y_1$. Conversely, let $f:\C^m\to Y_1$ be dominant and consider $b^{-1}\circ f:\C^m\dashrightarrow Y_2$. The valuative criterion for the proper morphism $b$ shows that this rational lift is defined at every codimension-one point of $\C^m$. Its indeterminacy locus $Z$ therefore has codimension at least two. If $Z$ is empty, the lift is already the required dominant morphism. Otherwise $m\ge2$, and $\C^m\setminus Z$ is flexible by \cite[Thm.~1.1]{FKZ16}. Restriction of a dominating spray to a fibre, followed by restriction to a transverse affine $m$-plane, gives a dominant morphism $g:\C^m\to\C^m\setminus Z$. Thus $b^{-1}\circ f\circ g$ is a dominant morphism to $Y_2$.
\end{remark}

\begin{corollary}\label{grp:complete-simply-connected}
Complete algebraically elliptic manifolds $Y$ are simply connected.
\end{corollary}
\begin{proof}
The restriction of a dominating spray to a fibre is a dominant morphism from an affine space to $Y$, so $Y$ is unirational. By Chow's lemma \cite[Ex.~II.4.10]{Har77} and resolution of singularities \cite[Thm.~1.1]{BM08}, $Y$ has a smooth projective birational model, which is unirational and hence simply connected by Serre's theorem \cite{Ser59}. Its fundamental group surjects onto $\pi_1(Y)$, since the birational morphism is an isomorphism over a dense Zariski open subset of the algebraic manifold $Y$.
\end{proof}

\subsection{Algebraic approximations to classifying spaces}

Configuration spaces in higher dimensions give increasingly connected maps to the classifying spaces of finite groups.

\begin{corollary}[highly connected maps to classifying spaces]\label{cor:algebraic-bg}
\leavevmode
\begin{enumerate}[label=\textup{(\roman*)}]
\item For every finite group $G$ and $d\ge1$, there is an algebraically
elliptic quasi-affine variety with a $d$-connected map to $K(G,1)$.
\item If a group $G$ has $K(G,1)$ homotopy equivalent to an algebraically
elliptic variety, then $G=1$.
\item For every nonzero abelian group $G$ and $n\ge2$, no complex
algebraic variety is homotopy equivalent to $K(G,n)$.
\end{enumerate}
\end{corollary}
\begin{proof}
\textup{(i)} Embed $G\le\mathfrak S_n$, choose $r\ge2$ with
$2r-1\ge d$, and put $F=\Conf_n(\C^r)$ and $Y=F/G$.
Proposition~\ref{prop:flexible-config} makes $Y$ algebraically
elliptic. The connectivity calculation preceding
Theorem~\ref{thm:elliptic-finite-groups} shows that $F$ is
$(2r-2)$-connected. For a contractible free $G$-CW complex $EG$, the
Borel fibration is
\[
F\longrightarrow EG\times_G F\longrightarrow BG=K(G,1).
\]
Since the action on $F$ is free, $EG\times_G F\simeq Y$. Thus the
classifying map $Y\to K(G,1)$ has homotopy fibre $F$ and is
$(2r-1)$-connected.

\textup{(ii)} If an algebraically elliptic variety realizes $K(G,1)$,
then $G$ is finite by Proposition~\ref{prop:serre-affine-finite}.
A group with a finite-dimensional $K(G,1)$ has finite integral
cohomological dimension and is torsion-free
\cite[Cor.~VIII.2.5]{Bro82}, so $G=1$.

\textup{(iii)} Complex algebraic varieties have the homotopy type of
finite CW complexes \cite[Ch.~1, Cor.~6.10]{Dim92}. If such a variety
realized $K(G,n)$ with $G\ne0$ and $n\ge2$, then
Corollary~\ref{cor:higher-em} would give $G=\Q^s$ for some finite
$s\ge1$. But $H_n(K(G,n);\Z)=G$ would also be finitely generated,
a contradiction.
\end{proof}

By Morgan's mixed Hodge theory, some finite simply connected complexes
are not homotopy equivalent to any algebraic manifold
\cite[\S9]{Mor78}, although they have Oka realizations. The rational
Sullivan minimal model $(\Lambda V,d)$ of a simply connected
algebraic manifold admits a weight decomposition compatible with
its cohomological grading:
\[
V=\bigoplus_{w\ge1}V_{(w)},\qquad
V_{(w)}=\bigoplus_{q\ge2}(V_{(w)}\cap V^q).
\]
Assigning a product the sum of the weights of its factors, one has
\[
d(V_{(w)})\subset(\Lambda V)_{(w)}.
\]
This is the condition of \emph{positive weights}; the rational
decomposition follows from Morgan's descent argument
\cite[\S10, especially Thm.~10.1]{Mor78}; see also
\cite[\S3]{Man22}.

\begin{problem}\label{prob:algebraic}
Which finite CW complexes with finite fundamental group are homotopy equivalent to algebraically elliptic manifolds? In particular, does every finite simply connected CW complex with positive weights have such a realization?
\end{problem}

\subsection{Compact K\"ahler Oka manifolds}\label{sec:kahler}

We realize every virtually abelian K\"ahler group as the fundamental group of a projective Oka manifold. The torsion-free virtually abelian K\"ahler groups occur as fundamental groups of quotients of abelian varieties by free actions of finite groups. For groups with torsion, we also use simply connected projective Oka manifolds with free actions of finite groups, obtained by resolving finite quotients of powers of an elliptic curve. To prove that these resolutions are Oka, we use \cite[Thm.~1.7]{ADGWX26}: a connected complex manifold is Oka if the complement of a proper closed analytic subset is Oka.

We fix the functorial resolution of Bierstone--Milman for complex spaces, denoted by $\operatorname{Res}$ \cite[Thm.~1.1 and \S1]{BM08}. For a reduced complex space $Z$, $\operatorname{Res}(Z)\to Z$ is a composition of blowups with smooth centres, finitely many over each relatively compact open set, and all centres lie over $\operatorname{Sing}Z$, so it is an isomorphism over $Z\setminus\operatorname{Sing}Z$. The resolution commutes with smooth holomorphic maps, in particular with biholomorphisms. If $Z$ is projective, all centres are compact, hence algebraic by Chow's theorem \cite{Ser56}, and $\operatorname{Res}(Z)$ is a projective manifold.

\begin{theorem}[free actions of every prime order]\label{thm:prime-actions}
For every prime $p$, there is a simply connected projective Oka manifold $S_p$ of complex dimension $p(p-1)$ with a free holomorphic $\Z/p$ action.
\end{theorem}

\begin{proof}
\emph{Setup.} Let $E=\C/\Lambda_E$ be an elliptic curve and let $r$ cyclically permute the coordinates of
\[
A_0=\{x\in E^p:\textstyle\sum_i x_i=0\}\cong E^{p-1}.
\]
On its universal cover $V_0=\{v\in\C^p:\sum_i v_i=0\}$, the eigenvalues of $r$ are the nontrivial $p$th roots of unity; its fixed points in $A_0$ form $B_0=\{(t,\dots,t):t\in E[p]\}$, where $E[p]=\{t\in E:pt=0\}$. Put $A=A_0^p$, $V=V_0^p$, $N=p(p-1)$, and let $R$ act by $r$ in every block, indexed by $j\in\Z/p$. Every $R^m$, $0<m<p$, has fixed point set $B=B_0^p\subset A$.
The group $\langle R\rangle$ acts freely on $A\setminus B$, and the quotient $Q=A/\langle R\rangle$ is projective and smooth off the image $\bar B$ of $B$. Writing $A=V/\Lambda$, put $\Pi=\Lambda\rtimes\langle R\rangle$, so $Q=V/\Pi$.

\emph{The free action.} Choose $a=(t,\dots,t)\in B_0$ with $t\ne0$, and define
\[
\alpha(x_j)_j=(r^jx_j+a)_j .
\]
Then $\alpha$ commutes with $R$, since $ra=a$, and $\alpha^m(x)_j=r^{jm}x_j+ma$, so $\alpha$ has order $p$. No $\alpha^m$, $0<m<p$, fixes a point of $Q$: otherwise $\alpha^mx=R^\ell x$ for some $x$ and $\ell$, and in the block $j$ with $jm\equiv\ell$ this reads $r^\ell x_j+ma=r^\ell x_j$, contradicting $ma\ne0$.

\emph{The resolution.} Take $\rho:S_p=\operatorname{Res}(Q)\to Q$. Since $Q$ is projective, $S_p$ is a projective manifold of dimension $N$, and $\rho$ is an isomorphism over $Q\setminus\bar B$. Functoriality lifts $\alpha$ to an automorphism $\hat\alpha$ of order $p$, and its action is free because a fixed point would map to a fixed point in $Q$.

\emph{Simple connectivity.} The group $\Pi$ acts properly discontinuously on the simply connected space $V$, and it is generated by elements with fixed points, since $z\mapsto Rz+\lambda$ has a fixed point and the translation by $\lambda$ is its composition with $R^{-1}$. Hence $Q=V/\Pi$ is simply connected by Armstrong's theorem \cite{Arm68}. Since $N\ge2$, the finite quotient map $A\to Q$ is \'etale in codimension one. The pair $(Q,0)$ is Kawamata log terminal \cite[Prop.~5.20(4)]{KM98}, so the resolution $\rho$ induces an isomorphism $\pi_1(S_p)\cong\pi_1(Q)$ \cite[Thm.~1.1]{Tak03}. Thus $\pi_1(S_p)=1$.

\emph{The Oka property.} Forgetting the last coordinate identifies $A_0$ with $E^{p-1}$ as a group, and hence $A$ with $E^N$. Under this identification, $B\subset E[p]^N$, so
\[
W=(E^2\setminus E[p]^2)\times E^{N-2}\subset A\setminus B.
\]
The surface $E^2\setminus E[p]^2$ is covered by
$\C^2\setminus(\tfrac1p\Lambda_E)^2$. By \cite[Prop.~4.1]{BL00},
a holomorphic automorphism of $\C^2$ maps this lattice into the
line $\{z_2=0\}$. The complement of the lattice is therefore Oka by
\cite[Thm.~1.6(a)]{FP02}.
The Oka property descends under holomorphic coverings and is preserved
by products \cite[\S5.6]{For17}, so $W$ is Oka.
The complement
\[
(A\setminus B)\setminus W
=\bigl(E[p]^2\times E^{N-2}\bigr)\setminus B
\]
is a proper closed analytic subset of the connected manifold
$A\setminus B$. Hence $A\setminus B$ is Oka by
\cite[Thm.~1.7]{ADGWX26}, and so is its free quotient
$Q\setminus\bar B$ \cite[Prop.~5.6.3]{For17}.

Finally, $S_p$ is connected and $\rho^{-1}(\bar B)$ is a proper closed
analytic subset with
$S_p\setminus\rho^{-1}(\bar B)\cong Q\setminus\bar B$.
A second application of \cite[Thm.~1.7]{ADGWX26} shows that $S_p$ is Oka.
\end{proof}

For $p=2$, we have $A=E\times E$ and $R=-1$. The functorial resolution of $\C^2/\{\pm1\}$ is its minimal resolution: blowing up the singular point gives the smooth total space of $\mathcal O_{\CP^1}(-2)$, and the lifted $\GL_2(\C)$ action is transitive on the exceptional curve. Functoriality therefore leaves no subsequent point centre, while blowing up the exceptional Cartier divisor is an isomorphism. Thus $S_2$ is the Kummer surface of $E\times E$, and $S_2/\langle\hat\alpha\rangle$ is an Enriques surface. Kummer surfaces were already known to be strongly dominable \cite[Cor.~7.2.3]{For17}.

\begin{lemma}[coinduced finite group actions]\label{lem:induced-actions}
Let $H$ be a subgroup of a finite group $F$, and let $S$ be a simply connected compact K\"ahler Oka manifold with a free holomorphic $H$ action. Then $S^{F/H}$ admits a holomorphic $F$ action in which no nonidentity element of $H$ has a fixed point.
\end{lemma}

\begin{proof}
Choose representatives $r_c$ of the cosets $c\in F/H$ with $r_H=1$. For $g\in F$ put
\[
h(g,c)=r_{gc}^{-1}gr_c\in H,\qquad
(g\cdot x)_{gc}=h(g,c)\cdot x_c.
\]
The cocycle identity $h(g_1g_2,c)=h(g_1,g_2c)h(g_2,c)$ shows that this is an action. For $g\in H$, the coordinate indexed by $c=H$ transforms by the given free action.
\end{proof}

\begin{proposition}[free actions of finite groups]\label{grp:finite-actions}
Every finite group $F$ acts freely and holomorphically on a simply
connected projective Oka manifold $S_F$ of complex dimension at most
$|F|(|F|-1)$.
\end{proposition}
\begin{proof}
For each subgroup $H\le F$ of prime order,
Theorem~\ref{thm:prime-actions} and Lemma~\ref{lem:induced-actions}
give an $F$ action on $S_{|H|}^{F/H}$ in which $H$ acts without
fixed points. Put
\[
S_F=\prod_{\substack{H\le F\\ |H|\text{ prime}}}S_{|H|}^{F/H},
\]
taking $S_F$ to be a point if $F=1$. Every $g\ne1$ in $F$ has a power $h$ of prime order,
and the factor for $\langle h\rangle$ shows that $g$ has no fixed point.
The dimension of this product is
\[
\begin{aligned}
\dim_{\C}S_F
&=\sum_{\substack{H\le F\\ |H|\text{ prime}}}
  [F:H]\,|H|(|H|-1)\\
&=|F|\cdot\#\{g\in F:\operatorname{ord}(g)\text{ is prime}\}
\le |F|(|F|-1),
\end{aligned}
\]
since the nonidentity elements of the prime-order subgroups partition
the set of prime-order elements of $F$.
\end{proof}

For $F\ne1$, neither $S_F$ nor $S_F/F$ is unirational. Indeed,
$S_F/F$ is a smooth projective variety \cite[Lem.~3.3]{BR11}
with fundamental group $F$,
so Serre's theorem \cite{Ser59} excludes unirationality. A dominant
rational map from a projective space to $S_F$ would give one to $S_F/F$
by composition with the quotient map. Neither variety is therefore
algebraically elliptic, since restricting a dominating algebraic spray
to a fibre gives a dominant morphism from an affine space.

\begin{theorem}[virtually abelian K\"ahler groups]\label{thm:va-full}
For a finitely generated virtually abelian group $\Gamma$ with a normal free abelian subgroup $\Lambda$ of finite index, the following are equivalent:
\begin{enumerate}[label=(\roman*)]
\item $\Gamma$ is the fundamental group of a compact K\"ahler Oka manifold;
\item $\Gamma$ is the fundamental group of a projective Oka manifold;
\item $\Gamma$ is a K\"ahler group;
\item there is an $\R$-linear map
$J:\Lambda\otimes\R\to\Lambda\otimes\R$ with $J^2=-\mathrm{id}$ and
$J\rho(g)=\rho(g)J$ for every $g\in\Gamma/\Lambda$, where $\rho$ is the
action induced by conjugation in $\Gamma$.
\end{enumerate}
When these equivalent conditions hold, the manifold in \textup{(ii)}
can be chosen of complex dimension at most
\[
\frac{\operatorname{rank}\Lambda}{2}+|F|(|F|-1),
\qquad F=\Gamma/\Lambda.
\]
It can be chosen aspherical if and only if $\Gamma$ is torsion-free.
In that case, a quotient of an abelian variety by a free action of a
finite group gives such a manifold of complex dimension
$\operatorname{rank}\Lambda/2$.
\end{theorem}

\begin{proof}
The implications (ii)$\Rightarrow$(i)$\Rightarrow$(iii) are immediate, and (iii)$\Leftrightarrow$(iv) is \cite[Thm.~1.1]{BR11}; Johnson proved the corresponding characterization for flat K\"ahler manifolds \cite{Joh91}; see \cite[\S1.3]{BR11}. Assume (iv), and put $F=\Gamma/\Lambda$. By \cite[Thm.~1.9]{BR11}, choose such a $J$ for which the complex torus $(\Lambda\otimes\R,J)/\Lambda$ is an abelian variety. Baues and Riesterer then construct a homomorphism $a:\Gamma\to\C^k\rtimes U(k)$, $k=\operatorname{rank}\Lambda/2$, whose restriction to $\Lambda$ is injective and has image a lattice of translations \cite[Lem.~3.1]{BR11}. Identifying $\Lambda$ with this image, $T=\C^k/\Lambda$ is an abelian variety. If $\Gamma$ is torsion-free, it acts freely on $\C^k$ and $\C^k/\Gamma$ is a projective manifold with $\pi_1=\Gamma$ \cite[\S3, proof of Cor.~1.7]{BR11}; it is aspherical, and Oka since it is covered by $\C^k$ \cite[Prop.~5.6.3]{For17}.

In general, take $S_F$ with its free holomorphic $F$ action from Proposition~\ref{grp:finite-actions}. Thus $Y=S_F/F$ is a projective manifold \cite[Lem.~3.3]{BR11} with $\pi_1(Y)=F$ and universal cover $S_F$. The construction of \cite[\S3]{BR11} applied to $Y$ gives a projective manifold
\[
X=(\C^k\times S_F)/\Gamma
\]
with $\pi_1(X)=\Gamma$, where
$\gamma\cdot(z,s)=(a(\gamma)z,\bar\gamma\cdot s)$ and $\bar\gamma$
is the image of $\gamma$ in $F$ \cite[Prop.~3.2]{BR11}.
The action is free: if $\bar\gamma\ne1$, it has no fixed point on $S_F$;
if $\gamma\in\Lambda\setminus\{0\}$, it is a nonzero translation on
$\C^k$. It is properly discontinuous because the lattice subgroup
$\Lambda$ has finite index. Thus $\C^k\times S_F$ is the universal
cover of $X$.
Since $\Lambda$ acts trivially on $S_F$, taking the quotient by $\Lambda$
gives
\[
X\cong(T\times S_F)/F,
\]
where the action of $F$ on $T$ is induced by the affine action of
$\Gamma$ on $\C^k$.
The quotient map $T\times S_F\to X$ is a finite unramified holomorphic
covering. Since $T\times S_F$ is Oka, so is $X$ \cite[Prop.~5.6.3]{For17}.
Moreover, $\dim_{\C}X=k+\dim_{\C}S_F$, which gives the stated bound.
If a manifold with fundamental group $\Gamma$ is aspherical, it gives
a finite-dimensional $K(\Gamma,1)$. Hence $\Gamma$ has finite integral
cohomological dimension and is torsion-free \cite[Cor.~VIII.2.5]{Bro82}.
Baues and Riesterer also prove that a virtually abelian K\"ahler group
is the fundamental group of an aspherical compact K\"ahler manifold
if and only if it is torsion-free \cite[Cor.~1.8]{BR11}.
\end{proof}

\begin{proof}[Proof of Theorem~\ref{intro:projective}]
Recall that a compact K\"ahler manifold $X$ is \emph{special} if
every rank-one coherent subsheaf $\mathcal L\subset\Omega_X^p$,
$p>0$, has Kodaira--Iitaka dimension $\kappa(X,\mathcal L)<p$
\cite[preprint, Def.~2.23 and Thm.~2.26]{Cam04}.
By \cite[Prop.~7.1.9(a)]{For17}, every Oka manifold is strongly dominable. Campana's criterion \cite[preprint, Cor.~8.11]{Cam04} therefore makes a compact K\"ahler Oka manifold special. Theorem~1.1 of the preprint \cite{OAI26} gives a virtually abelian fundamental group for every special compact K\"ahler manifold; the threefold case is due to Campana--Claudon \cite{CC14}. Conversely, every virtually abelian K\"ahler group is finitely generated and hence has a normal free abelian subgroup of finite index. Theorem~\ref{thm:va-full} therefore realizes it as the fundamental group of a projective Oka manifold.
\end{proof}

For comparison, if $\Gamma\le\SL_2(\C)$ is a torsion-free cocompact lattice, the compact homogeneous Oka manifold $\SL_2(\C)/\Gamma$ has fundamental group $\Gamma$, which is not virtually abelian. Hence this manifold is not K\"ahler.

For a finite group $G\ne1$, every compact K\"ahler manifold $X$ with
$\pi_1(X)\cong G$ has $\pi_2(X)\ne0$. Indeed, its universal cover
$\widetilde X$ is a simply connected compact K\"ahler manifold of
positive dimension. The Hurewicz theorem and the nonzero K\"ahler class
give $\pi_2(X)\cong H_2(\widetilde X;\Z)\ne0$. Consequently, such an
$X$ admits no $3$-connected map to $K(G,1)$.

\section{Cofibrations, acyclic embeddings, and holomorphic factorizations}
\label{sec:factorizations}\label{cof:section}

A holomorphic map is an \emph{Oka map} if it is a Serre fibration and
satisfies the parametric Oka property with approximation and interpolation for liftings
\cite[Def.~7.4.7]{For17}. A Serre fibration is a continuous map with the
homotopy lifting property for disks $D^n$, $n\ge0$.
L\'arusson's fine intermediate model structure relates the Oka principle to
homotopical algebra: the presheaves associated with Stein manifolds are
cofibrant, those associated with Oka manifolds are fibrant, and Oka maps
induce fibrations \cite{Lar04,Lar05}.
We first characterize the holomorphic maps that are cofibrations, and
then consider proper holomorphic embeddings that preserve homotopy type.

\subsection{Intermediate cofibrations}

For the ambient model category, let $\mathcal S$ be L\'arusson's
simplicially enriched category of second-countable Stein manifolds,
including disconnected manifolds, with mapping simplicial sets
$\operatorname{Sing}\mathcal O(P,Q)$ \cite[\S2]{Lar04}.
Here $\operatorname{Sing}$ is the singular simplicial set functor.
An object of $\mathcal S$ may have connected components of different
dimensions, with no uniform dimension bound. The manifolds $A$ and $X$
in the statements below each have a fixed finite complex dimension.
A complex manifold $Y$ determines the simplicially enriched presheaf
\[
 h_Y(P)=\operatorname{Sing}\mathcal O(P,Y),\qquad P\in\mathcal S,
\]
where spaces of holomorphic maps carry the compact-open topology.
The ambient category consists of contravariant simplicial functors from
$\mathcal S$ to the category $\mathbf{sSet}$ of simplicial sets.
A holomorphic map $i:A\to X$ is an
\emph{intermediate cofibration} if $h_i:h_A\to h_X$ is a cofibration
in the intermediate model structure.
A \emph{Stein inclusion} is a closed holomorphic embedding into a Stein
manifold.

\begin{theorem}[holomorphic cofibrations]\label{cof:main}
Let $i:A\to X$ be a holomorphic map of complex manifolds. Then $i$ is
an intermediate cofibration if and only if it is a biholomorphism or a
Stein inclusion. The classification holds for both the fine and the
coarse intermediate model structures.
\end{theorem}

This answers the question following Theorem~7.1 of
\cite[p.~32]{FL11}, also stated in L\'arusson's appendix to
\cite[Theorem~4.1 and the subsequent discussion, p.~802]{For13}.
The proof has two parts: a holomorphic cofibration is a closed embedding,
and its target is Stein unless the map is a biholomorphism.
Following \cite[proof of Theorem~6]{Lar05}, we pass to simplicial
degree $0$. With the notation below, we show that every pullback
$H_P\times_{Z_\alpha}H_A$ is $H_C$ for a closed reduced complex subspace
$C\subset P$. We then glue
holomorphic maps on a countable Stein coordinate cover.

The generating family for intermediate cofibrations consists of
\begin{equation}\label{cof:generators}
 (h_S\times\partial\Delta^n)
 \mathop{\amalg}_{h_T\times\partial\Delta^n}
 (h_T\times\Delta^n)
 \longrightarrow h_S\times\Delta^n,
 \qquad n\geq0,
\end{equation}
where $T\hookrightarrow S$ is a Stein inclusion, together with
the projective generators
$h_S\times\partial\Delta^n\to h_S\times\Delta^n$
\cite[arXiv v4, \S12 and Theorem~13]{Lar04}.
Here $\Delta^n$ is the standard simplicial $n$-simplex and
$\partial\Delta^n$ is its boundary.
Intermediate cofibrations are retracts of relative cell complexes
formed from these maps by pushouts and transfinite compositions.
Colimits are computed pointwise and degreewise.
In simplicial degree $0$, every generator with $n\geq1$ is an
isomorphism. For $n=0$, the generators in \eqref{cof:generators} give
$\mathcal O(\,\cdot\,,T)\hookrightarrow\mathcal O(\,\cdot\,,S)$,
and the projective generators give
$0\to\mathcal O(\,\cdot\,,S)$, where $0$ is the initial presheaf.

Let $\mathcal S_+$ be the full subcategory of the underlying ordinary
category of $\mathcal S$ on nonempty Stein manifolds. After taking
simplicial degree $0$, we restrict presheaves to $\mathcal S_+$. For a complex manifold or
reduced complex space $Y$, put
\[
 H_Y(P)=\mathcal O(P,Y),\qquad P\in\mathcal S_+.
\]
At the empty object, the empty-manifold representable and the initial
presheaf satisfy
\[
 (h_\varnothing)_0(\varnothing)=\{*\},\qquad
 0(\varnothing)=\varnothing.
\]
Both restrict to the initial presheaf $0$ on $\mathcal S_+$.
Thus the generators with $n=0$ have the form
$H_T\hookrightarrow H_S$ with $T\ne\varnothing$, or $0\to H_S$.
Restriction to $\mathcal S_+$ preserves the pointwise colimit
calculations. For $P\in\mathcal S_+$, the Yoneda lemma for set-valued presheaves gives
\begin{equation}\label{cof:yoneda}
 \operatorname{Nat}(H_P,F)\cong F(P).
\end{equation}

For the proof of Theorem~\ref{cof:main}, fix an intermediate cofibration
$i:A\to X$. The relative-cell data introduced below remain fixed until
that proof is completed. In the category of simplicially enriched
presheaves, factor $h_i$ as
$h_A\xrightarrow{k}\mathcal Z\xrightarrow{r}h_X$, where $k$ is a
transfinite composition of pushouts of generating cofibrations and $r$
is an acyclic fibration.
The lifting axiom supplies $j:h_X\to\mathcal Z$ with
$rj=\mathrm{id}_{h_X}$ and $jh_i=k$.
Taking simplicial degree $0$ and then restricting to $\mathcal S_+$ gives a
relative cell complex $k:H_A\to Z$ and identities
\begin{equation}\label{cof:retract-identities}
 rj=\mathrm{id}_{H_X},\qquad ji_*=k,\qquad rk=i_*,
 \qquad i_*=H_i.
\end{equation}
Well-order the cells attached at each stage and attach them successively,
using their original attaching maps. This gives a transfinite sequence
\begin{equation}\label{cof:filtration}
 Z_0=H_A\longrightarrow Z_1\longrightarrow\cdots
 \longrightarrow Z_\lambda=Z.
\end{equation}
In simplicial degree $0$, each successor map is either a pushout of
$H_T\hookrightarrow H_S$ (including $0\to H_S$) or an isomorphism.
At a limit ordinal, the presheaf is the pointwise colimit of the earlier
stages.
Each map $Z_\alpha\to Z_\beta$, $\alpha\le\beta$, is a monomorphism
because the generators are monomorphisms. Write
$k_\alpha:H_A=Z_0\to Z_\alpha$ for the structure map.
In particular, $ji_*=k$ implies that $i_*$ is a monomorphism;
evaluation at a point shows that $i$ is injective.

We will use two elementary observations about these cells.
If $D\subset C$ are sets and $D\to B$ is any map, then
\begin{equation}\label{cof:set-pushout}
 B\amalg_D C=B\amalg(C\setminus D).
\end{equation}
Indeed, the pushout relations replace every element of the copy of
$D$ in $C$ by its image in $B$, and involve no element of
$C\setminus D$. Consequently, $B$ and $C\setminus D$ inject into the pushout with disjoint images. For $A_0\subset B$, this also gives
\begin{equation}\label{cof:set-pullback}
 C\times_{B\amalg_D C}A_0=D\times_B A_0.
\end{equation}
These identifications apply pointwise to presheaves.

\begin{lemma}\label{cof:reduced-pullback}
Let $g:P\to S$ be a holomorphic map of complex manifolds and
$D\hookrightarrow S$ a closed reduced complex subspace. Then
$C=(P\times_S D)_{\mathrm{red}}$ is a closed reduced complex subspace
of $P$ and satisfies
\[
 H_C=H_P\times_{H_S}H_D.
\]
Moreover, a holomorphic map from a reduced complex space with image
in a closed reduced complex subspace factors holomorphically through
that subspace.
\end{lemma}

\begin{proof}
A holomorphic function on a reduced space that vanishes at every
point is zero. If a holomorphic map from a reduced complex space has image in a closed
complex subspace, the pullback of every local section of its defining
ideal vanishes at every point and hence is zero. A map from a reduced space to an
analytic fibre product also annihilates its nilpotent functions, and
hence factors through its reduction. Applying the fibre-product
universal property to every Stein manifold $Q$ gives the natural
bijections
\[
 \mathcal O(Q,C)=
 \mathcal O(Q,P)\times_{\mathcal O(Q,S)}\mathcal O(Q,D).
\]
\end{proof}

\subsection{Closed embeddings from relative cells}

\begin{proposition}\label{cof:analytic-inverse-images}
For every stage $Z_\alpha$, every $P\in\mathcal S_+$, and every
map $v:H_P\to Z_\alpha$, there are a closed reduced complex subspace
$c:C\hookrightarrow P$ and a holomorphic map $a:C\to A$ such that
\[
\begin{array}{ccc}
 H_C&\xrightarrow{H_a}&H_A\\
 \big\downarrow{\scriptstyle H_c}&&
 \big\downarrow{\scriptstyle k_\alpha}\\
 H_P&\xrightarrow{v}&Z_\alpha
\end{array}
\]
is a pullback. The subspace $C$ may be empty.
\end{proposition}

\begin{proof}
We construct $C$ and $a:C\to A$ by transfinite induction on $\alpha$.
For $\alpha=0$, the Yoneda lemma identifies $v:H_P\to H_A$ with a
holomorphic map $P\to A$, and we take $C=P$.

Consider a successor stage
\begin{equation}\label{cof:successor}
 Z_{\alpha+1}=Z_\alpha\amalg_{H_T}H_S,
 \qquad T\hookrightarrow S\text{ a Stein inclusion}.
\end{equation}
By \eqref{cof:yoneda}, the map $v$ is determined by
$v_P(\mathrm{id}_P)\in Z_{\alpha+1}(P)$. If this element lies in $Z_\alpha(P)$,
the induction hypothesis applies, and the pullback is unchanged
because $Z_\alpha\to Z_{\alpha+1}$ is a monomorphism.

Otherwise, $v$ is represented by a holomorphic map $g:P\to S$.
For $T\ne\varnothing$, apply the induction hypothesis to the
attaching map $H_T\to Z_\alpha$. It gives a closed reduced complex
subspace $D\subset T$ and a holomorphic map $b:D\to A$.
Formula \eqref{cof:set-pullback}, evaluated on each $Q\in\mathcal S_+$, gives
\begin{equation}\label{cof:cell-pullback}
 H_S\times_{Z_{\alpha+1}}H_A
 =H_T\times_{Z_\alpha}H_A=H_D,
\end{equation}
compatibly with the maps to $H_S$ and $H_A$; the latter is $H_b$.
The composite $D\hookrightarrow T\hookrightarrow S$ identifies $D$
with a closed reduced complex subspace of $S$. Put
\[
 C=(P\times_S D)_{\mathrm{red}},\qquad
 a:C\longrightarrow D\xrightarrow{b}A.
\]
Lemma~\ref{cof:reduced-pullback} and
\eqref{cof:cell-pullback} yield
\[
 H_P\times_{Z_{\alpha+1}}H_A
 =H_P\times_{H_S}H_D=H_C,
\]
with the required holomorphic maps. For a cell $0\to H_S$, the
pullback of $H_S\to Z_{\alpha+1}$ along $H_A\to Z_{\alpha+1}$ is
empty, so take $C=\varnothing$. If the successor map is an isomorphism
in simplicial degree $0$, use the preceding stage.

At a limit stage $\gamma$, the element $v_P(\mathrm{id}_P)\in Z_\gamma(P)$
belongs to some earlier $Z_\alpha(P)$. The induction hypothesis
there gives $C$ and $a$. Since $Z_\alpha\to Z_\gamma$ is a monomorphism,
$H_P\times_{Z_\gamma}H_A=H_P\times_{Z_\alpha}H_A=H_C$.
\end{proof}

\begin{corollary}\label{cof:closed-embedding}
Every holomorphic intermediate cofibration is a closed holomorphic
embedding.
\end{corollary}

\begin{proof}
Let $P\subset X$ be a nonempty Stein coordinate open set, and apply
Proposition~\ref{cof:analytic-inverse-images} to
$H_P\to H_X\xrightarrow{j}Z$.
The identities \eqref{cof:retract-identities} give, for every $Q\in\mathcal S_+$,
\[
 j_Q(x)=k_Q(a)\quad\Longleftrightarrow\quad x=i_{*,Q}(a).
\]
Thus the resulting closed reduced complex subspace $C\subset P$ satisfies
\begin{equation}\label{cof:local-pullback}
 H_C=H_P\times_ZH_A
 =H_P\times_{H_X}H_A=H_{A_P},
 \qquad A_P=i^{-1}(P).
\end{equation}
The proposition supplies a holomorphic map $a:C\to A$, and applying
$r$ to its commutative square shows that $ia$ is the inclusion of $C$
in $P$. Hence $a$ is holomorphic as a map $C\to A_P$.
Evaluating \eqref{cof:local-pullback} on the one-point Stein manifold gives
$i(A_P)=C$ and identifies $a$ as the pointwise inverse of
$i|_{A_P}$.
By Lemma~\ref{cof:reduced-pullback}, the map $i|_{A_P}:A_P\to P$
factors holomorphically as $b:A_P\to C$.
The composites $ab$ and $ba$ equal the identity on points, and
therefore as holomorphic maps of reduced spaces. It follows that
$A_P\to C$ is a biholomorphism and that
$i|_{A_P}$ is a closed embedding. Closed embeddings are local on
the target, so these coordinate opens give the conclusion.
\end{proof}

\subsection{The Stein target condition}

A countable disjoint union of Stein manifolds is an object of $\mathcal S$. Indeed, it is second countable; holomorphic functions
are chosen componentwise, and locally constant functions separate
the components. A compact set meets only finitely many components,
and its holomorphic hull is the finite union of its compact hulls
in those components. Every second-countable complex manifold has
a countable Stein coordinate cover, whose disjoint union is therefore an object of $\mathcal S$.

\begin{proposition}\label{cof:stein-target}
If $i:A\to X$ is a holomorphic intermediate cofibration, then $i$
is a biholomorphism or $X$ is Stein.
\end{proposition}

\begin{proof}
The case $X=\varnothing$ gives $A=\varnothing$. Otherwise choose a
countable cover $(U_n)_{n\in I}$ of $X$ by nonempty Stein coordinate
balls and put
\[
 U=\coprod_{n\in I}U_n,\qquad q:U\to X.
\]
The restriction of $q$ to the component $U_n$ is the inclusion
$U_n\hookrightarrow X$, and $U\in\mathcal S_+$. Consider
\begin{equation}\label{cof:atlas-vertex}
 z=j_U(q)\in Z(U).
\end{equation}

If $z$ belongs to $Z_0(U)=H_A(U)$, it is represented by a
holomorphic map $a:U\to A$. Applying $r$ gives $ia=q$.
Since $i$ is injective, the restrictions of $a$ agree on chart
overlaps and glue to a holomorphic map $b:X\to A$ with
$ib=\mathrm{id}_X$. Injectivity also gives $bi=\mathrm{id}_A$.

If $z\notin Z_0(U)$, let $\beta+1$ be the least ordinal with
$z\in Z_{\beta+1}(U)$. This is a successor ordinal because limit stages
are pointwise unions. At this stage the pushout has the form
\begin{equation}\label{cof:birth}
 Z_{\beta+1}=Z_\beta\amalg_{H_T}H_S,
\end{equation}
where $H_T$ may be $0$.
A holomorphic map $g:U\to S$ represents $z$ in this pushout.
If $T\ne\varnothing$, then $g$ does not factor through $T$.
Lemma~\ref{cof:reduced-pullback} therefore gives an index $k$ and a
point $v\in U_k$ with $g(v)\notin T$; put $V=U_k$.
For a cell $0\to H_S$, choose any $v\in U$ and let $V$ be the
component containing $v$.

Let $W$ be a nonempty Stein coordinate ball in an overlap
$U_m\cap U_n$. Let $e_m,e_n:W\amalg V\to U$ restrict on $W$ to its inclusions in
the components $U_m$ and $U_n$, respectively, and both restrict on $V$
to the inclusion $V\hookrightarrow U$. Then
\[
 qe_m=qe_n.
\]
Naturality of $j$ in \eqref{cof:atlas-vertex} gives equality in
$Z(W\amalg V)$ of the elements represented by $ge_m$ and $ge_n$.
Both maps lie in
\[
 H_S(W\amalg V)\setminus H_T(W\amalg V),
\]
since their common restriction to $V$ takes $v$ outside $T$.
By \eqref{cof:set-pushout}, this complement injects into
$Z_{\beta+1}(W\amalg V)$; the map from this stage to each later stage is also a monomorphism.
The two maps are consequently equal as holomorphic maps to $S$.
Thus the restrictions $g|_{U_n}$ agree on overlaps and glue to a
holomorphic map
\[
 s:X\to S,\qquad sq=g.
\]

The composite $H_S\to Z\xrightarrow{r}H_X$ is induced, by the Yoneda lemma, by a holomorphic map $p:S\to X$.
Applying $r$ to $g$ gives $pg=q$; surjectivity of $q$ now yields
\begin{equation}\label{cof:holomorphic-retract}
 ps=\mathrm{id}_X.
\end{equation}
Thus $s$ is injective, has injective differential, and has continuous
inverse $p$ on its image. It is a holomorphic embedding, and
\[
 s(X)=\{y\in S:(sp)(y)=y\}
\]
is closed. Hence $X$ is biholomorphic to a closed complex
submanifold of the Stein manifold $S$ and is Stein.
\end{proof}

\begin{proof}[Proof of Theorem~\ref{cof:main}]
Corollary~\ref{cof:closed-embedding} and
Proposition~\ref{cof:stein-target} give the forward implication.
Conversely, a Stein inclusion is a generating cofibration at $n=0$,
and a biholomorphism induces an isomorphism of enriched presheaves.
The fine and coarse intermediate structures have the same
cofibrations \cite[``Factorization in the intermediate structure'']{Lar05}.
\end{proof}

\begin{remark}\label{cof:disconnected-probes}
For nonempty manifolds $A$ and $B$, a holomorphic map
$\{1,2\}\to A\amalg B$ can take its two points into different summands.
Hence $H_{A\amalg B}\ne H_A\amalg H_B$ on $\mathcal S_+$.
On the full subcategory of nonempty connected Stein manifolds these
presheaves coincide. There, $H_A\to H_{A\amalg S}$ is a pushout of
$0\to H_S$ for every Stein manifold $S$. If $A$ is non-Stein and $S$
is nonempty, this map is not an isomorphism and its target is non-Stein.
The proof above evaluates presheaves on the disconnected Stein manifolds
$U$ and $W\amalg V$.
Theorem~\ref{cof:main} also holds for second-countable $A$ and $X$ whose
connected-component dimensions have no uniform bound.
\end{remark}

\subsection{Proper Stein embeddings into Oka open sets}

We apply Proposition~\ref{tw:oka} to a tubular neighbourhood of a
properly embedded Stein manifold.

\begin{theorem}\label{thm:stein-open-embedding}
Every nonempty Stein manifold $S$ admits a proper holomorphic embedding
$j:S\hookrightarrow D$ into an Oka open set $D\subset\C^N$ such that
$j$ is a homotopy equivalence.
More precisely, a proper holomorphic embedding
$e:S\hookrightarrow\C^{2q}$, $q\ge1$, gives such an embedding into
an Oka open set $D\subset\C^{2q+1}$.
\end{theorem}
\begin{proof}
By \cite{Nar60,EG92}, there is a proper holomorphic embedding of $S$
into a complex Euclidean space. Composing with a linear inclusion gives
$e:S\hookrightarrow\C^{2q}$ for some $q\ge1$.
The image is a closed smooth submanifold. Choose an open smooth tubular
neighborhood $U\subset\C^{2q}$ of $e(S)$, so that $e(S)$ is a strong
deformation retract of $U$~\cite[Theorem~6.24]{Lee13}.
Proposition~\ref{tw:oka} gives the Oka open set
$D_U\subset\C^{2q+1}$. The map
\[
 j:S\longrightarrow D_U,\qquad j(s)=(e(s),0),
\]
is a holomorphic injective immersion. For a compact set $K\subset D_U$,
the set $j^{-1}(K)$ is closed and is contained in the compact set
$e^{-1}(\operatorname{pr}_w K)$, so $j$ is proper.
The inverse of the diffeomorphism \eqref{tw:diffeomorphism} takes $j(s)$ to
$(e(s),0)\in U\times\C$. The tubular retraction and
contraction of the second factor show that $j$ is a homotopy
equivalence.
\end{proof}

\subsection{Acyclic embeddings and Oka bundles}

In the fine intermediate model structure, a holomorphic map of complex
manifolds is a weak equivalence if and only if it is a homotopy
equivalence \cite[\S4]{Lar04}. Following \cite[\S4]{Rit13}, an
\emph{acyclic embedding} is a proper holomorphic embedding that is a
homotopy equivalence.
Theorem~\ref{cof:main} then gives
\begin{equation}\label{cof:acyclic-criterion}
 \begin{gathered}
 \text{an acyclic embedding is a fine intermediate acyclic cofibration}\\
 \Longleftrightarrow\quad
 \text{it is a biholomorphism or its target is Stein}.
 \end{gathered}
\end{equation}

By Theorem~\ref{thm:stein-open-embedding}, every nonempty Stein manifold
has an acyclic embedding into an Oka open subset of a complex
Euclidean space. This answers the question about embeddings into Oka
manifolds stated in \cite[\S1.2]{DeV20}. Ritter proved that every open
Riemann surface admits an acyclic embedding into an elliptic manifold
\cite[Theorem~3]{Rit13}.
The relative construction below also controls a given map to an
aspherical manifold.

\begin{theorem}[factorizations through Oka maps]\label{thm:factorization}
Let $Y$ be a connected complex manifold.
\begin{enumerate}[label=(\roman*)]
\item If $Y$ is aspherical, then every holomorphic map $f:X\to Y$
from a nonempty Stein manifold factors as $f=p\circ\iota$, where
$\iota:X\to Z$ is an acyclic embedding and $p:Z\to Y$ is a
holomorphic fibre bundle with Oka fibre. In particular, $p$ is an
Oka map. If $Y$ is Oka, then so is $Z$.
\item If there is a holomorphic Serre fibration $Z\to Y$ with $Z$
contractible, then $Y$ is aspherical.
\end{enumerate}
\end{theorem}

\begin{proof}
(i) Put $\pi=\pi_1(Y)$ and let $q:\widetilde Y\to Y$ be the
contractible universal cover. Theorem~\ref{intro:countable} gives
an Oka manifold $W$ and a homotopy equivalence $h:X\to W$.
Choose a continuous map $\varphi:W\to Y$ with
$\varphi h\simeq f$.
The pullback $E=W\times_Y\widetilde Y$ of the covering $q$ along
$\varphi$ carries the complex structure induced by its covering
map to $W$. The covering manifold $E$ is Oka
\cite[Proposition~5.6.3]{For17}, and $\pi$ acts freely by
holomorphic deck transformations. Set
\[
 Z_0=(\widetilde Y\times E)/\pi,\qquad
 p_0[\tilde y,e]=q(\tilde y),
\]
using the diagonal action. On an evenly covered open set $V\subset Y$,
a holomorphic local section $\sigma:V\to\widetilde Y$ gives the
trivialization
\[
 V\times E\xrightarrow{\sim}p_0^{-1}(V),\qquad
 (y,e)\longmapsto[\sigma(y),e].
\]
Thus $Z_0$ is a complex manifold and $p_0:Z_0\to Y$ is a holomorphic
fibre bundle with Oka fibre $E$, hence an Oka map
\cite[Corollary~7.4.8(i)]{For17}.

The other projection $a:Z_0\to E/\pi=W$ is a bundle with
contractible fibre $\widetilde Y$, and hence a homotopy equivalence.
The maps $p_0$ and $\varphi a$ from $Z_0$ to $Y$ are homotopic.
Indeed, both factor continuously through
\[
 Z_0\longrightarrow(\widetilde Y\times\widetilde Y)/\pi,
 \qquad
 [\tilde y,(w,\tilde y')]\longmapsto[\tilde y,\tilde y'],
\]
and then through the two projections to $Y$.
Triangulate $(\widetilde Y\times\widetilde Y)/\pi$ and lift the
resulting CW structure to $\widetilde Y\times\widetilde Y$.
This is a free $\pi$-CW structure. Since $\widetilde Y$ is contractible,
a $\pi$-equivariant homotopy between the two projections extends over
each cell, with the two projections prescribed at the endpoints.
The induced maps $(\widetilde Y\times\widetilde Y)/\pi\to Y$ are
therefore homotopic. Choose a continuous homotopy inverse
$b:W\to Z_0$ of $a$ and put $g=bh$. Then $g:X\to Z_0$ is a homotopy
equivalence and
\[
 p_0g\simeq\varphi abh\simeq\varphi h\simeq f.
\]

Since $p_0$ is a Serre fibration and $X$ is triangulable, homotopy
lifting deforms $g$ to a map $g'$ satisfying $p_0g'=f$.
The parametric Oka property of $p_0$ and the Stein property of $X$ then deform
$g'$, through lifts of $f$, to a holomorphic lift $g''$
\cite[Definition~7.4.7]{For17}.
Pair $g''$ with a proper holomorphic embedding $j:X\to\C^m$,
as in \cite[proof of Lemma~1]{Rit13}, and put
\[
 Z=Z_0\times\C^m,\qquad
 \iota=(g'',j),\qquad
 p=p_0\circ\operatorname{pr}_{Z_0}.
\]
The map $\iota$ is a proper injective immersion, hence a closed
holomorphic embedding. It is a homotopy equivalence because $g''$
is, and $p\iota=f$.
The map $p$ is a holomorphic fibre bundle with Oka fibre
$E\times\C^m$. If $Y$ is Oka, then $Z$ is Oka by \cite[Theorem~5.6.5]{For17}.

(ii) Write $p:Z\to Y$ for the given Serre fibration.
It is surjective by path lifting. For $y_0\in Y$, the fibre
$F=p^{-1}(y_0)$ is a closed analytic subset of $Z$.
Its underlying topological space has a triangulation
of real dimension at most $2\dim_{\C}Z$
\cite[\S3, Theorem~2]{Loj64}. Hence $H_j(F;K)=0$ for
$j>2\dim_{\C}Z$ and every field $K$.
Since $Z$ is contractible, $F$ is weakly equivalent to the based loop
space $\Omega Y$ at $y_0$, whose components are homotopy equivalent
to $\Omega\widetilde Y$.
It follows that $H_*(\Omega\widetilde Y;K)$ vanishes in large
degrees. The same is true of $H_*(\widetilde Y;K)$, since
$\widetilde Y$ is a finite-dimensional manifold.

Apply the homology Serre spectral sequence over $K$ to the path
fibration of the simply connected space $\widetilde Y$. If $a$ and
$b$ are the largest degrees with nonzero homology of $\widetilde Y$
and $\Omega\widetilde Y$, respectively, then
\[
 E^2_{a,b}=H_a(\widetilde Y;K)\otimes_K
 H_b(\Omega\widetilde Y;K)\ne0
\]
has no incoming or outgoing differential. The total path space is
contractible, so $a=b=0$. Thus
$\widetilde H_*(\widetilde Y;K)=0$ for every field $K$
\cite[Chapter~III, corollary to Proposition~3]{Ser51}.
The universal coefficient theorem, applied with $\Q$ and all
$\mathbb F_p$, gives $\widetilde H_*(\widetilde Y;\Z)=0$
\cite[Chapter~III, no.~3, Lemme]{Ser51}.
The Hurewicz and Whitehead theorems show that $\widetilde Y$ is
contractible.
\end{proof}

Suppose that a holomorphic map $f:X\to Y$ factors
as $f=p\iota$, where $\iota:X\to Z$ is an acyclic embedding and
$p:Z\to Y$ is a holomorphic Serre fibration. For $y\in Y$, the maps
$\iota$ and $\mathrm{id}_Y$ induce a weak equivalence from the homotopy
fibre of $f$ over $y$ to that of $p$ over $y$. The fibre $p^{-1}(y)$
is weakly equivalent to the homotopy fibre of $p$ over $y$, and its
underlying topological space has a
finite-dimensional triangulation by \cite[\S3, Theorem~2]{Loj64}.
Since $X$ and $Y$ are triangulable, the homotopy fibre of $f$
has CW type \cite[p.~408]{Hat02}. By Whitehead's theorem, it
therefore has the homotopy type of a finite-dimensional CW complex.

\begin{corollary}\label{cof:aspherical-factorization}
Let $X$ be a nonempty Stein manifold, $Y$ a connected complex manifold,
and $f:X\to Y$ a nullhomotopic holomorphic map. The following conditions
are equivalent:
\begin{enumerate}[label=(\roman*)]
\item $Y$ is aspherical.
\item $f$ factors as an acyclic embedding followed by an Oka map.
\item $f$ factors as an acyclic embedding followed by a holomorphic
Serre fibration.
\end{enumerate}
\end{corollary}
\begin{proof}
The implication (i)$\Rightarrow$(ii) is Theorem~\ref{thm:factorization}(i),
and (ii)$\Rightarrow$(iii) follows from the definition of an Oka map.
Assume (iii). By the preceding paragraph, the homotopy fibre of $f$
has the homotopy type of a finite-dimensional CW complex. Choose the
endpoint of a nullhomotopy of $f$ as the basepoint of $Y$, and write
$\operatorname{hofib}(f)$ for the homotopy fibre over this point.
The nullhomotopy gives
\[
 \operatorname{hofib}(f)\simeq X\times\Omega Y.
\]
Choose $x_0\in X$. The inclusion $\omega\mapsto(x_0,\omega)$ and the
second projection make $\Omega Y$ a retract of $X\times\Omega Y$.
Thus $H_j(\Omega Y;K)=0$ in sufficiently large degrees for every field
$K$. The component of the constant loop is homotopy equivalent to
$\Omega\widetilde Y$, and its homology is a direct summand of
$H_*(\Omega Y;K)$. Hence $H_j(\Omega\widetilde Y;K)$ also vanishes in
large degrees. The manifold $\widetilde Y$ has real dimension
$2\dim_{\C}Y$, so $H_j(\widetilde Y;K)=0$ for $j>2\dim_{\C}Y$.
Applying the homology
Serre spectral sequence and then the universal coefficient, Hurewicz
and Whitehead theorems as in the proof of
Theorem~\ref{thm:factorization}(ii) shows that $\widetilde Y$ is
contractible.
\end{proof}

By Corollary~\ref{cof:aspherical-factorization}, no constant map from
a nonempty Stein manifold to $\CP^1$ factors as an acyclic embedding
followed by a holomorphic Serre fibration.

An acyclic embedding $X\hookrightarrow Z$ with $X$ Stein and $Z$ Stein
Oka represents a fibrant replacement of $h_X$ in the fine intermediate
model structure: $h_X\to h_Z$ is an acyclic cofibration and $h_Z$ is
fibrant. Ritter asks whether every Stein manifold admits an acyclic
embedding into a Stein elliptic manifold \cite[\S4]{Rit13}; see also
L\'arusson's discussion of fibrant models
\cite[\S4.4, p.~803]{For13}.

For Stein manifolds, the Oka property and ellipticity are equivalent
\cite[Theorem~2]{Lar05}. Moreover, if a Stein manifold $X$ is homotopy
equivalent to a Stein Oka manifold $W$, the Oka principle makes a chosen homotopy
equivalence $X\to W$ holomorphic. Pairing this map with a proper
holomorphic embedding $X\hookrightarrow\C^m$ gives an acyclic
embedding into the Stein Oka manifold $W\times\C^m$.
By Remark~\ref{rem:stein-types}, every countable finite-dimensional
CW complex has a Stein realization. Thus Ritter's question is
equivalent to the following problem.

\begin{problem}\label{prob:stein-oka}
Is every countable finite-dimensional CW complex homotopy
equivalent to a Stein Oka manifold?
\end{problem}

\subsection*{Acknowledgements}
Yuta Kusakabe thanks Franc Forstneri\v{c}, Shulim Kaliman, Finnur L\'arusson and Riccardo Ugolini for comments on his earlier preprint \cite{Kus22}, and Shulim Kaliman for pointing out Morgan's work \cite{Mor78}. His work was supported by JSPS KAKENHI Grant Numbers JP21K20324 and JP24K16919.

\subsection*{Use of generative AI}
Generative AI tools were used to assist with mathematical arguments, literature searches, and the preparation and revision of the manuscript.

\end{document}